\documentclass{amsart}
\usepackage{amssymb}
\usepackage{amsmath}
\usepackage{mathrsfs}
\usepackage{amsfonts}
\usepackage{amscd}
\usepackage{pifont}
\usepackage{txfonts}
\usepackage{dsfont}
\usepackage{color}

\newtheorem{theorem}{Theorem}[section]
\newtheorem{proposition}[theorem]{Proposition}
\newtheorem{lemma}[theorem]{Lemma}
\newtheorem{corollary}[theorem]{Corollary}

\theoremstyle{definition}

\newtheorem{remark}[theorem]{Remark}
\theoremstyle{remark}

\numberwithin{equation}{section}

\begin{document}

\title[Hessian equations]{Local solvability\\
of the $k$-Hessian equations}

\author[G. Tian, Q. Wang \& C.-J. Xu]{G. Tian, Q. Wang and C.-J. Xu}

\address{Guji Tian, Wuhan Institute of Physics and
Mathematics,Chinese Academy of
 Sciences,430072 Wuhan, P.R. China,}
\email{tianguji@wipm.ac.cn}

\address{Qi Wang,  School of Mathematics and Statistics, Wuhan University 430072,
Wuhan, P. R. China}
\email{qiwang88@whu.edu.cn}

\address{Chao-Jiang Xu,  School of Mathematics and Statistics, Wuhan University 430072,
Wuhan, P. R. China\\
Universit\'e de Rouen, CNRS UMR 6085, Laboratoire de Math\'ematiques, 76801 Saint-Etienne du Rouvray, France}
\email{chao-jiang.xu@univ-rouen.fr}

\date{\today}



\maketitle

\begin{abstract}
In this work, we study the existence of local solutions in $\mathbb{R}^{n}$ to $k$-Hessian equation,
for which the nonhomogeneous term $f$ is permitted to change the sign or be non negative; if  $f$ is $C^\infty,$  so is the local solution. We also give a classification for the second order polynomial solutions to the $k-$Hessian equation, it is the basis to construct the local solutions and obtain the uniform ellipticity of the linearized operators at such constructed local solutions.

\end{abstract}

\section{Introduction}
In this paper, we focus on the existence of local  solution for the following $k$-Hessian equation,
\begin{equation}\label{eq:1.2}
S_{k}[u]=f(y,u,Du)\, ,
\end{equation}
on an open domain $\Omega$ of $\mathbb{R}^{n}$, where $2\leq k\le n$.
For a smooth function  $u\in C^2$, the $k$-Hessian operator $S_{k}$ is defined by
\begin{equation}\label{eq:1.1}
S_{k}[u]=S_k(D^2u)=\sigma_{k}(\lambda(D^{2}u))=\sum_{1\leq i_{1}<
i_{2}\ldots<i_{k}\leq
n}\lambda_{i_{1}}\lambda_{i_2}\ldots\lambda_{i_{k}},
\end{equation}
where $\lambda(D^{2}u)=(\lambda_{1},\ldots,\lambda_{n})$  are the
eigenvalues of the Hessian matrix $(D^{2}u)$, $\sigma_{k}(\lambda)$ is the $k-$th
elementary symmetric polynomial, and $S_k[u]$ is the sum of all principal minors of
order $k$ for the Hessian matrix $(D^2u)$ . We say that a smooth function $u$ is $k$-convex
if  the eigenvalues of the Hessian matrix $(D^{2}u)$ are in the so-called G{\aa}rding cone $\Gamma_{k}$ which is defined by
$$
\Gamma_{k}(n)=\{\lambda=(\lambda_1,\lambda_2,\ldots,\lambda_n)\in\mathbb{R}^{n}:
\sigma_j(\lambda)>0,1\leq j\leq k\}.
$$

For the local solvability, Hong and Zuily \cite{HZ} obtained the existence of $C^{\infty}$ local solutions for
Monge-Amp\'{e}re equation
\begin{equation}\label{eq:1.4}
\det D^{2}u=f(y,u,Du)\indent y\in \Omega\subset\mathbb{R}^{n},
\end{equation}
when $f\in C^\infty$ is nonnegative, it is the case of $k=n$ in \eqref{eq:1.2}. The geometric background of
Monge-Amp\`ere equation can be found in \cite{HH,GS,Lin1}.  In this work, we only consider the Hessian
equation for  $2\leq k< n,$ since it is classical for k=1. We will follow the method
of  \cite{HZ} (see also  \cite{Lin,TWX}) to construct the local solution by a  perturbation of
the polynomial-typed solution of $S_{k}[u]=c$ for some real constant $c$. Since the right hand side function in
\eqref{eq:1.2}  possibly vanishes,  then, the solution is in the closure of $\Gamma_{k}$. Thus, we need to study the closure of $\Gamma_{k}$, its boundary is
$$
\partial\Gamma_{k}(n)=\{\lambda\in \mathbb{R}^{n}: \sigma_{j}(\lambda)\geq0
, \sigma_{k}(\lambda)=0,1\leq j\leq k-1\}.
$$
From the Maclaurin's inequalities
$$
\left[\frac{\sigma_{k}(\lambda)}{\big(^{n}_{k}\big)}\right]^{1/k}\leq
\left[\frac{\sigma_{l}(\lambda)}{\big(^{n}_{l}\big)}\right]^{1/l},\ \
\lambda\in \Gamma_{k}, \,k\geq l\geq 1,
$$
we see that $\sigma_{k+1}(\lambda)>0$
cannot occur for $\lambda\in\partial \Gamma_{k}(n)$, therefore
we can express $\partial\Gamma_{k}$ as two parts
$$
\partial\Gamma_{k}(n)= \mathbf P_{1}\cup \mathbf P_{2},
$$
with
\begin{align*}
& \mathbf P_{1}=\{\lambda\in \Gamma_{k}(n):\sigma_{j}(\lambda)\geq0,\,\sigma_{k}(\lambda)=\sigma_{k+1}(\lambda)=0,1\leq j\leq k-1\}\\
& \mathbf P_{2}=
\{\lambda\in \Gamma_{k}(n):\sigma_{j}(\lambda)\geq0,\,\sigma_{k}(\lambda)=0,\,\sigma_{k+1}(\lambda)<0,1\leq j\leq k-1\}.
\end{align*}
Besides,
\begin{equation}\label{eq:1.7}
\mathbf{P}_{2}=\emptyset, \ \  \text{if} \ \ k=n.
\end{equation}
In Section 2, we will prove that $\mathbf P_{1}$ and $\mathbf P_{2}$ have a more precise version, and $\mathbf P_{2}\neq \emptyset$ if $k<n.$

In this paper, we always discuss the $k$-Hessian equation under the framework of ellipticity,  then we follow the ideas of \cite{Ivo} and \cite{IPY} to get the existence of the solution. Here we explain the ellipticity: in the matrix language, the ellipticity set of the $k$-Hessian operator, $1\le k\le n,$ is
$$
E_k=\left\{S\in \mathcal{M}_s(n) : S_{k}(S+t\,\xi\times\xi)>S_{k}(S)>0, \xi\in \mathbb{S}^{n-1}, \forall  t\in \mathbb{R}^+\right\},
$$
where $\mathcal{M}_s(n)$ is the space of $n$-symmetric real matrix.  Then the G{\aa}rding cones is
$$
\Gamma_k=\left\{S\in \mathcal{M}_s(n): S_{k}(S+t\, \mathbf{I})>S_{k}(S)>0, \forall t\in \mathbb{R}^+\right\}.
$$
It is possible to show that $E_k=\Gamma_k$ only for $k=1,n$ and the
example  in \cite{IPY} assures that $\Gamma_k\subset E_k$ and
$mess(E_k\setminus \Gamma_k)>0.$ Ivochkina, Prokofeva and Yakunina
\cite{IPY} pointed out that the ellipticity of \eqref{eq:1.2} is
independent of the sign of $f$ if $k<n$. But for the Monge-Amp\`ere equation \eqref{eq:1.4}, the type of equation is
determined by the sign of $f$, it is elliptic  if $f>0$,
hyperbolic if $f<0$ and degenerate if $f$ vanishes at some points; it is of mixed
type if $f$ changes sign \cite{Han}.

There are many results for the Dirichlet problem of \eqref{eq:1.2} under the condition $f> 0$
(see \cite{W} and references therein), there are also some results about $C^{1, 1}$ weak solution of the Dirichlet problem
of \eqref{eq:1.2} under the degenerate condition $f\geq 0$ (see \cite{WX} and references therein).
But  similarly to Monge-Amp\`ere equation, the existence of the smooth solution to Dirichlet problem of \eqref{eq:1.2} is completely
an open problem if $f$ is not strictly positive.

In this work, for the  local solution of the $k$ -Hessian equation \eqref{eq:1.2}, we prove the following results.

\begin{theorem}\label{th:Main1}
Let $f=f(y,u,p)$ be defined and continuous near a point $Z_0=(y_0, u_0, p_0)\in \mathbb{R}^n\times\mathbb{R}\times\mathbb{R}^{n}$,  $0<\alpha<1, 2\le k<n$. Assume that $f$ is $C^\alpha$ with respect to $y$ and  $C^{2,1}$ with respect to $u,p.$   We have that
\begin{itemize}
 \item [(1)] If $f(Z_0)=0$, then the equation \eqref{eq:1.2}  admits  a $(k-1)$ - convex local solution $u\in
C^{2, \alpha}$ near $y_0$  which is not $(k+1)$ - convex.

\item [(2)]  If $f\geq 0$ near $Z_0$, then the equation \eqref{eq:1.2}  admits  a $k$ -convex local solution $u\in
C^{2, \alpha}$ near $y_0$  which  is not $(k+1)$ - convex.

\item [(3)]  If $f(Z_0)<0$,  the equation \eqref{eq:1.2} admits  a $(k-1)$-convex local solution $u\in
C^{2, \alpha}$ near $y_0$  which  must not be $k$ - convex.
\end{itemize}
Moreover, in all the case above,  the linearized operator of \eqref{eq:1.2} at $u$ is uniformly elliptic,
and if $f\in C^{\infty}$ near $Z_0$,  then the local solution above is $C^\infty$ near $y_0$.
\end{theorem}

\begin{remark} 1) Without loss of generality, by a translation $y\rightarrow y-y_0$ and replacing
$u$ by $u-u(0)-y\,\cdot\,Du(0),$ we can assume $Z_0=(0, 0, 0)$ in Theorem \ref{th:Main1},
then the local solution in the above Theorem  \ref{th:Main1} is of  the following form
\begin{equation}\label{eq:1.5}
  u(y)=\frac{1}{2}\sum^{n}_{i=1}\tau_{i}y_{i}^{2}
+\varepsilon'\varepsilon^{4}w(\varepsilon^{-2}y),
\end{equation}
with  arbitrarily fixed $(\tau_{1},\tau_{2},\ldots,\tau_{n})\in \mathbf{P}_{2}$ in the cases of (1) and (2).
In the case of (3), we take some special $ (\tau_{1},\tau_{2},\ldots,\tau_{n})\in \Gamma_{k-1}$.
In \eqref{eq:1.5}, we always take $\varepsilon>0$  and
\begin{equation}\label{eq:1.5BB}
\varepsilon'=\left\{\begin{array}{ll}
\varepsilon^\alpha, &\quad 0<\alpha\le \frac{1}{2}\\
\varepsilon, &\quad  \frac 12<\alpha<1,\end{array}\right.
\end{equation}
then \eqref{eq:1.5} is of the same form as the solution in \cite{HZ,Lin,TWX}, where $f$ has good smoothness.

2) Notice that, in Case (1) of  Theorem \ref{th:Main1},  $f$ is permitted to change sign near $Z_0$.

3) If $f=f(y,u,p)$ is independent of $u$ and $p$, then the assumption on $f$ is reduced to $f=f(y)\in C^\alpha$, which is a necessary requirement on $f$  for the classical Schauder theory.
The condition that $f$ is $C^{2,1}$ with respect to $u$ and $p$ is  a technical one to meet the need
for tackling with the quadratic error in Nash-Moser iteration (see \eqref{eq:7.15}).

4) In Theorem \ref{th:Main1}, we consider only the $k$-Hessian equation with $2\le k<n$, since the
Monge-Amp\`ere equation \eqref{eq:1.4} is considered by \cite{Han,HZ}.
\end{remark}

This article consists of three sections besides the introduction. In
Section \ref{section2}, we give  a classification of the polynomial-typed solutions for $S_{k}[u]=c$ for some real
constant $c$.  Such a  classification will assure the ellipticity of linearized operators at each polynomial. The results  of this section given also a good understanding for the structure of solutions to $k$- hessian equation. In Section \ref{section3}, Theorem \ref{th:Main1} is proved by Nash-Moser iteration. Section \ref{section4} is an appendix in which three equivalent definitions for G{\aa}rding cone
are given and proved.


\section{A classification of  polynomial solutions}\label{section2}

For $\lambda\in \mathbb{R}^n$, set $ \psi(y)=\frac{1}{2}\sum_{i=1}^n\lambda_iy_i^2$,  then
\begin{equation}\label{2.1}
S_k[\psi]=\sigma_k(\lambda)=c,
\end{equation}
where $c$ is a real constant. The linearized operators of $S_k[\,\cdot\,]$ at $\psi$ is
\begin{equation}\label{2.2}
\mathcal{L}_\psi=\sum_{i=1}^n\sigma_{k-1,i}(\lambda)\partial_i^2,
\end{equation}
where $\sigma_{k-1,i}(\lambda)$, furthermore, $\sigma_{l ,i_1, i_2, \cdots, i_s}(\lambda)$, is defined in \eqref{eq:2.10}. To give a classification of polynomial solutions to equation \eqref{2.1},  we recall a results of Section 2 of \cite{W}.

\begin{proposition}\label{prop2.1}(See \cite{W})
Assume that $\lambda\in\overline{\Gamma}_{k}(n)$ is in descending order,
\begin{itemize}
\item[(i)]  then we have
$$
\lambda_{1}\geq \cdots\lambda_{k}\geq \cdots\lambda_{p}\ge 0\geq
\lambda_{p+1}\geq \cdots\lambda_{n}
$$
with $p\geq k$.
\item[(ii)] we have
\begin{equation}\label{eq:2.17}
0\le\sigma_{k-1;1}(\lambda)\leq \sigma_{k-1;2}(\lambda)\leq
\cdots\leq \sigma_{k-1,n}(\lambda).
\end{equation}
\item[(iii)] For any $\{i_{1},i_{2},\cdots i_{s}\}\subset\{1,2,\ldots,n\}$ with $l+s\leq
k$, we have
\begin{equation}\label{eq:2.18}
\sigma_{l;i_{1},i_{2}\cdots i_{s}}(\lambda)\ge 0.
\end{equation}
\end{itemize}
\end{proposition}

Using (ii) of the Proposition \ref{prop2.1}, for any $\lambda\in\overline{\Gamma}_{k}(n)$, the linearized operators
$\mathcal{L}_\psi$ defined in \eqref{2.2}  could be  degenerate elliptic. Now we study the non-strictly
$k$-convex Garding's cone $\overline\Gamma_{k}(n)$, we will show some special uniformly elliptic case.

\begin{theorem}\label{1th}
Suppose that $\lambda\in \partial\Gamma_{k}(n)=\mathbf P_{1}\cup\mathbf P_{2}$, $2\leq k\leq n-1$. Then either
\begin{itemize}
\item[(I)]If $\sigma_{k}(\lambda)= 0$ and $\sigma_{k+1}(\lambda)< 0$, then
$$
\sigma_{k-1;i}(\lambda)>0,\, i=1,2,\cdots,n.
$$
\end{itemize}
or
\begin{itemize}
\item[(II)]If $\sigma_{k}(\lambda)=0=\sigma_{k+1}(\lambda)$, then $$
    \sigma_{j}(\lambda)=0,\qquad j=k+2,\cdots, n,
    $$
    that means $\lambda\in \overline\Gamma_{n}(n)$.
\end{itemize}
\end{theorem}

In order to prove this theorem, we need several lemmas.
\begin{lemma}\label{lm:3.4}
Let $\lambda=(\lambda_1,\cdots,\lambda_n)\in\mathbb{R}^n$ and it is in the descending order. For $0\le s< n-k-1$, denote  $\lambda^{(s)}=(\lambda_{s+1},\cdots,\lambda_n)$. Suppose that
$\lambda^{(s)}\in \overline\Gamma_{k}(n-s)$ and
$$
\left\{
\begin{array}{l}
\sigma_{k-1; s+1}(\lambda^{(s)})=0,\\
\sigma_{k}(\lambda^{(s)})=0,\\
\sigma_{k+1}(\lambda^{(s)})<0.
\end{array}\right.
$$
Then $\lambda^{(s+1)}=(\lambda_{s+2},\cdots,\lambda_n)\in \overline\Gamma_{k}(n-s-1)$ and
$$
\left\{
\begin{array}{l}
\sigma_{k-1; s+2}(\lambda^{(s+1)})=0,\\
\sigma_{k}(\lambda^{(s+1)})=0,\\
\sigma_{k+1}(\lambda^{(s+1)})<0.
\end{array}\right.
$$
\end{lemma}

\begin{proof} It suffices to prove this lemma for $s=0$ since we can complete the proof by an induction on the length of $\lambda^{(s)}$. Here we call the length of $\lambda^{(n-j)}$ is $j$ if $\lambda^{(n-j)}=(\lambda_{n-j+1},\ldots,\lambda_{n})$. Thus, we suppose that
\begin{equation}\label{eq:3.10}
\sigma_{k-1;1}(\lambda)=0,\quad \sigma_{k}(\lambda)=0,\quad \sigma_{k+1}(\lambda)<0.
\end{equation}
Using \eqref{eq:2.11}
$$
\sigma_{k}(\lambda)=\lambda_1\sigma_{k-1;1}(\lambda)+\sigma_{k;1}
(\lambda),
$$
and $$
\sigma_{k+1}(\lambda)=\lambda_1\sigma_{k;1}(\lambda)+\sigma_{k+1;1}
(\lambda),
$$
we get
\begin{equation}\label{eq:3.11}
\left\{
\begin{array}{l}
\sigma_{k-1;1}(\lambda_1,\cdots,\lambda_n)=\sigma_{k-1}(\lambda_2,
\cdots,\lambda_n)=0,\\
\sigma_{k;1}(\lambda_1,\cdots,\lambda_n)=\sigma_{k}(\lambda_2,
\cdots,\lambda_n)=0,\\
\sigma_{k+1;1}(\lambda_1,\cdots,\lambda_n)
=\sigma_{k+1}(\lambda_2,\cdots,\lambda_n)<0.
\end{array}\right.
\end{equation}
By \eqref{eq:2.18} and \eqref{eq:3.11},  we have, for
$\lambda\in \overline\Gamma_{k}(n)$ satisfying \eqref{eq:3.10},
$$
\sigma_{j;1}(\lambda_1,\cdots,\lambda_n)=\sigma_{j}
(\lambda_2,\cdots,\lambda_n)\geq 0, \quad\forall j\le k,
$$
which implies
\begin{equation}\label{eq:3.12}
\lambda^{(1)}:=(\lambda_{2},\lambda_{3},\cdots,\lambda_{n})\in \overline\Gamma_{k}(n-1).
\end{equation}
Using  Proposition \ref{prop2.1} for $\lambda^{(1)}\in
\overline\Gamma_{k}(n-1)$, we have
\begin{equation}\label{eq:3.13}
\lambda_2\geq\cdots\geq \lambda_{1+k}\geq 0,\quad \sigma_{k-1;2}(\lambda^{(1)})\geq0,\indent \sigma_{k-2;2}(\lambda^{(1)})\geq 0.
\end{equation}
Using the first equation in \eqref{eq:3.11} and the first and third inequalities in \eqref{eq:3.13}, we have
\begin{equation}\label{eq:3.14}
0=\sigma_{k-1;1}(\lambda)=\sigma_{k-1}(\lambda^{(1)})
=\lambda_{2}\sigma_{k-2;2}(\lambda^{(1)})+\sigma_{k-1;2}(\lambda^{(1)})\geq \sigma_{k-1;2}(\lambda^{(1)}).
\end{equation}
Then, by \eqref{eq:3.13} and \eqref{eq:3.14}
\begin{equation}\label{eq:3.15}
0\geq\sigma_{k-1;2}(\lambda^{(1)})\geq 0.
\end{equation}
Accordingly, from \eqref{eq:3.11}, \eqref{eq:3.12} and
\eqref{eq:3.15}, we get $\lambda^{(1)}\in \overline\Gamma_{k}(n-1)$
and
$$
\left\{
\begin{array}{l}
\sigma_{k-1; 2}(\lambda^{(1)})=0,\\
\sigma_{k}(\lambda^{(1)})=0,\\
\sigma_{k+1}(\lambda^{(1)})<0.
\end{array}\right.
$$
This completes the proof of Lemma \ref{lm:3.4} for $s=0.$ Then, by an induction on
 the length of $\lambda^{(s)}$, we finish the proof of Lemma \ref{lm:3.4}.
\end{proof}

\smallskip
By Proposition \ref{prop2.1}, if $\lambda\in \overline\Gamma_{m-1}(m), m>1$
and $\sigma_{m-1}(\lambda)=0,$ then
$$\sigma_{m-2,i}(\lambda)\geq0,\qquad i=1,2,\cdots,m.$$
Under additional condition $\sigma_{m}(\lambda)<0,$ we have

\begin{lemma}\label{lm:3.5}
Let $\lambda\in \overline\Gamma_{m-1}(m), m>2$  and
$\sigma_{m-1}(\lambda)=0.$ If $\sigma_{m}(\lambda)<0$, then we have
$$
\sigma_{m-2,i}(\lambda)>0,\qquad i=1,2,\cdots,m.
$$
\end{lemma}

\begin{proof}
By Proposition \ref{prop2.1}, we have $\sigma_{m-2,i}(\lambda)\geq 0$;  the Maclaurin's inequalities \eqref{eq:2.7} yields $\sigma_{m}(\lambda)\leq 0$.
Thus, we can equivalently say, if  the inequality above does not hold, that is,
$\sigma_{m-2, i}(\lambda)=0$ for some $i\in\{1, \cdots, m\}$, then
$$
\sigma_{m}(\lambda)=\lambda_1\cdots\lambda_m=0.
$$It is enough to prove that  $\sigma_{m-1}(\lambda)=0=\sigma_{m-2;1}(\lambda)=0$
imply $\sigma_{m}(\lambda)=0$, since the other case can be deduced by absurd argument of this results.

Substituting $\sigma_{m-1}(\lambda)=0=\sigma_{m-2;1}(\lambda)=0$ into

$$
\sigma_{m-1}(\lambda)=\lambda_1\sigma_{m-2;1}(\lambda)
+\sigma_{m-1;1}(\lambda),
$$
we have
$$
0=\sigma_{m-1;1}(\lambda_1,\cdots,\lambda_m )=\sigma_{m-1}(\lambda_2,\cdots,\lambda_m)=\prod_{i=2}^{m} \lambda_{i}.
$$
Thus,
$$
0=\lambda_{1}\prod_{i=2}^{m}\lambda_{i}=\sigma_{m}(\lambda).
$$
\end{proof}

\begin{proof}[{\bf Proof of Theorem \ref{1th}}]  By these two lemmas above, we prove Theorem \ref{1th} by an induction on $k$.
Let $\lambda=(\lambda_{1},\lambda_{2},\cdots,\lambda_{n})\in \mathbb{R}^{n}$ and $\lambda$ is in the descending order.

\noindent
\item{\bf Step 1: The case $k=2$.}  We claim the following results :

{\em Let $\lambda\in \overline\Gamma_{2}(n)$ and $\sigma_{2}(\lambda)=0$, if $\sigma_{3}(\lambda)<0$, then we have
$$
\sigma_{1,i}(\lambda)>0,\qquad i=1,2,\cdots,n.
$$
Equivalently, for $\lambda\in \overline\Gamma_{2}(n)$ with $\sigma_{2}(\lambda)=0$, if $\sigma_{1, i}(\lambda)=0$ for some $i\in\{1, \cdots, n\}$, then
$$
\sigma_{j}(\lambda)=0,\qquad j=3, \cdots, n.
$$
}

 By assumption above, we have
\begin{align*}
\sigma_{2}(\lambda)&=\lambda_{1}\sigma_{1;1}(\lambda)
+\sigma_{2;1}(\lambda),\\
\sigma_{1,1}(\lambda)&=\sigma_{1}(\lambda)-\lambda_1=\sum^n_{j=2}
\lambda_j=\sigma_{1}(\lambda_2, \cdots, \lambda_n),\\
\sigma_{2;1}(\lambda)&=\sigma_{2}(\lambda_2, \cdots, \lambda_n)
=\frac{(\sigma_{1}-\lambda_{1})^{2}-\sum_{i=2}^{m}\lambda^{2}_{i}}{2}.
\end{align*}
If we assume $\sigma_{1,1}(\lambda)=0$, then $\sigma_{1,1}(\lambda)=0$ together with $\sigma_{2}(\lambda)=0$ yields
$$
0=\sigma_{2}(\lambda)=\lambda_1\sigma_{1,1}(\lambda)+\sigma_{2,1}(\lambda)
=\sigma_{2,1}(\lambda)=\frac{(\sigma_{1}-\lambda_{1})^{2}-\sum_{i=2}^{m}\lambda^{2}_{i}}{2}
=\frac{-\sum_{i=2}^{m}\lambda^{2}_{i}}{2},$$
which implies
$$
\sum_{i=2}^{m}\lambda_{i}^{2}=0,
$$
and thus
$$
\lambda_{i}=0,\qquad i=2, \cdots, n.
$$
Then
$$
\sigma_{j}(\lambda)=0, \quad j=2, \cdots, n\, ,
$$
which contradicts with $\sigma_{3}(\lambda)<0$,  therefore $\sigma_{1,1}(\lambda)=0$ is impossible.

\item{\bf Step 2: The case $2< k\leq n-1.$ }
If $k=n-1$, it is included  in Lemma \ref{lm:3.5}. Now we consider  the general case $2< k< n-1$.

{\bf The proof of part (I):} We will prove that, for $2<k<n-1$,
if $\lambda\in \overline\Gamma_{k}(n)$, $\sigma_{k}(\lambda)=0$ and  $\sigma_{k+1}(\lambda)<0$, then
$$
\sigma_{k-1;1}(\lambda)>0 .
$$
We prove this claim by absurd argument. Recall $\lambda^{(s)}=(\lambda_{s+1},\ldots,\lambda_n)$ and suppose that
$\lambda=\lambda^{(0)}\in \overline\Gamma_{k}(n)$,
\begin{equation*}
\left\{
\begin{array}{l}
\sigma_{k-1; 1}(\lambda^{(0)})=0,\,\mbox{(the absurd assumption),}\\
\sigma_{k}(\lambda^{(0)})=0,\\
\sigma_{k+1}(\lambda^{(0)})<0.
\end{array}\right.
\end{equation*}
By using Lemma \ref{lm:3.4} and the induction assumption up to
$s=n-k-1$, we have that $\lambda^{(n-k-1)}\in
\overline\Gamma_{k}(k+1)$ and

\begin{equation*}
\left\{
\begin{array}{l}
\sigma_{k-1;n-k}(\lambda^{(n-k-1)})=\sigma_{k-1;n-k}(\lambda_{n-k},\cdots,
\lambda_{n})=0,\\
\sigma_{k}(\lambda^{(n-k-1)})=\sigma_{k}(\lambda_{n-k},\cdots,
\lambda_{n})=0,\\
\sigma_{k+1}(\lambda^{n-k-1})=\sigma_{k+1}
(\lambda_{n-k},\cdots,\lambda_{n})<0.
\end{array}\right.
\end{equation*}

This  contradicts with the conclusion of Lemma \ref{lm:3.5} with $m=k+1$. Thus,
the assumption $\sigma_{k-1;1}(\lambda)=0$ is really absurd, and therefore $\sigma_{k-1;1}(\lambda)>0$.

\smallskip

{\bf The proof of part (II):} We suppose that $\lambda\in \overline{\Gamma}_{k}(n)$ and
 $\sigma_{k}(\lambda)=\sigma_{k+1}(\lambda)=0$. Then $\lambda\in \overline{\Gamma}_{k+1}(n)$,
 and for any $\varepsilon>0$, from the formula
$$
\sigma_{k+1}(\lambda+\varepsilon)=
\sum_{j=0}^{k+1}C(j,k,n)\varepsilon^{j}\sigma_{k+1-j}(\lambda),\indent
\ C(j,k,n)=\frac{(^{n}_{k+1})(^{k+1}_{j})}{(^{n}_{k+1-j})}
$$
with the convention $\sigma_{0}(\lambda)=1$,  we have
$$
\lambda+\varepsilon=(\lambda_{1}+\varepsilon,\lambda_{2}
+\varepsilon,\cdots,\lambda_{n}+\varepsilon)\in \Gamma_{k+1}(n).
$$
We see that either $\sigma_{k+2}(\lambda)\geq 0$ or $\sigma_{k+2}(\lambda)< 0.$ Now we prove  $\sigma_{k+2}(\lambda)= 0.$ Firstly we claim that  $\sigma_{k+2}(\lambda)< 0$ is impossible. Otherwise, from the assumption $\sigma_{k}(\lambda)=\sigma_{k+1}(\lambda)=0$, we have
\begin{equation*}
\begin{split}
&\sigma_{k}(\lambda)=0,\\
&\sigma_{k+1}(\lambda)=0,\\
&\sigma_{k+2}(\lambda)<0.
\end{split}
\end{equation*}
Using the results of part (I), we obtain from both $\sigma_{k+1}(\lambda)=0$ and
$\sigma_{k+2}(\lambda)<0$  that
\begin{equation}\label{eq:3.17}
\sigma_{k; 1}(\lambda)>0.
\end{equation}
Since $\lambda\in \overline{\Gamma}_{k}(n),$ it is necessary by Proposition \ref{prop2.1} that
$\lambda_1\ge 0$ and $\sigma_{k-1;1}(\lambda)\ge 0$. So we have
$$
0=\sigma_{k}(\lambda)=\lambda_{1}
\sigma_{k-1;1}(\lambda)+\sigma_{k;1}(\lambda)\geq \sigma_{k;1}(\lambda).
$$
There is a contradiction of \eqref{eq:3.17}, thus, the case
that $\sigma_{k+2}(\lambda)< 0$ does not occur and we obtain  $\sigma_{k+2}(\lambda)\geq 0$ in which case $\lambda\in \overline\Gamma_{k+2}(n)$. Applying the Maclaurin's inequalities  to $\lambda+\varepsilon\in \Gamma_{k+2}(n)$, we obtain
$$
\left[\frac{1}{(^{n}_{k+1})}\sigma_{k+1}(\lambda+\varepsilon)
\right]^{\frac{1}{k+1}}\geq
\left[\frac{1}{(^{n}_{k+2})}\sigma_{k+2}(\lambda+\varepsilon)
\right]^{\frac{1}{k+2}}.
$$
Let $\varepsilon\rightarrow 0^{+}$, then
$$
\sigma_{k+1}(\lambda)=0\qquad
\mbox{implies}\qquad
\sigma_{k+2}(\lambda)=0.
$$

Repeating  the above argument for $k+1$, $k+2$,$\cdots$,
$n$, we obtain
$$
\sigma_{k+j}(\lambda)=0,\indent j=1,2,\cdots,n-k,
$$
that is, $\lambda\in \overline\Gamma_{n}(n)$ and this completes the proof.
\end{proof}

Now we are back to the definitions of $\mathbf{P}_{1}$ and $\mathbf{P}_{2}$,
by Theorem \ref{1th}, can be stated more precisely as, for $k<n$,
$$
\mathbf{P}_{1}=\{\lambda\in \Gamma_{k}:\sigma_{j}(\lambda)\geq 0,\,\sigma_{k}(\lambda)=\ldots=\sigma_{n}(\lambda)=0,\,1\leq j\leq k-1\},
$$
$$
\mathbf{P}_{2}=\{\lambda\in \Gamma_{k}:\sigma_{j}(\lambda)> 0,\,\sigma_{k}(\lambda)=0,\,\sigma_{k+1}<0,\,1\leq j\leq k-1\}.
$$
If $2\leq k<n$ , then $\mathbf{P}_{2}\neq\emptyset$.  Here is an example , let
\begin{equation*}
  \left\{
\begin{array}{l}
\lambda_i=1, 1\leq i\leq k-1 \\
\lambda_k=M,\lambda_{k+1}=-\frac{1}{M}\\
\lambda_i=0, k+1< i\leq n
\end{array}\right.
\end{equation*}
with $M=\frac{k-1+\sqrt{(k-1)^2+4}}{2}>1$, then $M-\frac{1}{M}=k-1,$ $\sigma_j(\lambda)>0 (1\leq j\leq k-1)$, $\sigma_k(\lambda)=0$
and $\sigma_{k+1}(\lambda)=-1$, which means that $\mathbf{P}_{2}\neq\emptyset.$

The significance of Theorem\ref{1th} is the breakthrough of the classical framework of the ellipticity of Hessian equations. It is well known that, $S_k[u]=f$ is elliptic if $f>0$  and degenerate elliptic if $f\geq 0.$ By the definition of $\mathbf{P}_2$, the condition $f(0)=0$ must lead to degenerate ellipticity for Monge-Amp\'{e}re equation. However, it is no longer true for $k-$Hessian ($2\leq k<n$) equation by Theorem \ref{1th}.  Next theorem gives a complete $k$-Hessian classification of second-order polynomials in the degenerate elliptic case.

\begin{theorem}\label{2th}
Suppose that $\lambda\in \partial\Gamma_{k}(n)$, $2\leq k\leq n-1$.
\begin{itemize}
\item[(1)] For any $\lambda\in \mathbf{P}_{2}$, we have that $ \psi=\frac{1}{2}\sum_{i=1}^n\lambda_iy_i^2 $
is a solution of $k-$Hessian equation $S_k[\psi]=0,$ and the linearized operators
of  $\mathcal{L}_\psi=\sum_{i=1}^n\sigma_{k-1,i}(\lambda)\partial_i^2 $ is uniformly elliptic.

\item[(2)]For any  $\lambda\in \mathbf{P}_{1}$ with $\lambda_1\geq\ldots \geq \lambda_n$, then $
\psi=\frac{1}{2}\sum_{i=1}^n\lambda_iy_i^2$
is a non-strict convex solution of $k-$Hessian equation $S_k[\psi]=0$ ,
and the linearized operators of $\mathcal{L}_\psi$ is degenerate elliptic with
$$
\sigma_{k-1,i}=0,\indent 1\leq i\leq k-1.
$$
Moreover, if $\sigma_{k-1}(\lambda)=0$, then $\mathcal{L}_\psi = 0$, that is, $\sigma_{k-1,i}=0$ for $1\leq i\leq n$; if
$\sigma_{k-1}(\lambda)>0,$ then
\begin{equation}\label{eq:3.5}
  \left\{
\begin{array}{l}
\lambda_i>0, 1\leq i\leq k-1 \\
  \lambda_i=0, k\leq i\leq n\\
  \sigma_{k-1,i}=0,1\leq i\leq k-1\\
  \sigma_{k-1,i}=\prod_{j=1}^{k-1}\lambda_j>0,k\leq i\leq n
\end{array}\right.
\end{equation}

\end{itemize}
\end{theorem}

\begin{proof}
(1) is a direct consequence of (I) in Theorem \ref{1th}. Now we prove (2). If $\lambda\in \mathbf{P}_{1}$, we have $\sigma_{k-1,i}\geq 0$ for $1\leq i\leq n$ by \eqref{eq:2.17},
and $\sigma_j(\lambda)=0$ for $k\leq j\leq n$ by Theorem \ref{1th} (II).
From $\sigma_n(\lambda)=\prod_{i=1}^n\lambda_i=0$, $\lambda_j\geq 0 (1\leq j \leq n)$ and
$\lambda$ is in decreasing order, it follows that $\lambda_n=0$ and
$$
\sigma_k(\lambda_1,\ldots,\lambda_{n-1})=
\sigma_k(\lambda)-\lambda_n\sigma_{k-1,n}(\lambda)=\sigma_k(\lambda)=0.
$$
Similarly,
$$
\sigma_{j}(\lambda_1,\ldots,\lambda_{n-1})=\sigma_j(\lambda)\geq0,
\sigma_{k+1}(\lambda_1,\ldots,\lambda_{n-1})=\sigma_{k+1}(\lambda)=0,1\leq j\leq k-1.
$$
Applying Theorem  \ref{1th} (II) to $(\lambda_1,\ldots,\lambda_{n-1})\in \Gamma_{k}(n-1),$
we obtain
$$
\sigma_{n-1}(\lambda_{1},\ldots,\lambda_{n-1})=0
$$
and by the same reasoning in the $n-$dimensional case, $\lambda_{n-1}=0.$ By an induction on the dimension up to $k+1,$ we see that
$\lambda_i=0, k+1\leq i\leq n.$ Since $\sigma_k(\lambda)=0$ and
$$
\sigma_k(\lambda)=\sigma_k(\lambda_1,\ldots,\lambda_k,0,\ldots,0)=\prod_{i=1}^k\lambda_i,
$$
 we have
$\lambda_k=0$ by recalling that $\lambda$ is in descending order. By the virtue of
$$
\sum_{i=1}^n\sigma_{k-1;i}(\lambda)=(n-k+1)\sigma_{k-1}(\lambda),\sigma_{k-1;i}(\lambda)\geq 0,\indent 1\leq i\leq n,
$$
we see that, if $\sigma_{k-1}(\lambda)=0$, then $\sigma_{k-1,i}= 0$ for $1\leq i\leq n$;
if $\sigma_{k-1}(\lambda)>0,$ from
$$
0<\sigma_{k-1}(\lambda)=\sigma_{k-1}(\lambda_1,\ldots,\lambda_{k-1},0,\ldots,0)
=\prod_{i=1}^{k-1}\lambda_i,
$$
we have
$$
\lambda_i>0,1\leq i\leq k-1.
$$
By the definition of $\sigma_{k-1;i}(\lambda)$ and
$$
\lambda=(\lambda_1,\ldots,\lambda_{k-1},0,\ldots,0),
$$
 we obtain the last two conclusions of \eqref{eq:3.5}.
\end{proof}

If $\lambda\in \Gamma_k(n)$, we certainly have $\lambda\in \Gamma_{k-1}(n).$
If $\lambda\in \mathbf{P}_{2}\subset\partial \Gamma_k(n)$ with $\sigma_{k+1}(\lambda)<0$,
from $\sum_{i=1}^n\sigma_{k-1;i}(\mu)=(n-k+1)\sigma_{k-1}(\mu),\mu\in \mathbb{R}^n$ and Theorem  \ref{1th}  (I), it follows that $\sigma_{k-1}(\lambda)>0$ and $\lambda\in \Gamma_{k-1}(n).$
 In those two cases above, $0<\sigma_{k-1;1}\leq \sigma_{k-1;2}\leq\ldots\leq \sigma_{k-1;n}$ implies the uniform ellipticity. Conversely, we want to know whether and how the ellipticity  is true for $\lambda\in \Gamma_{k-1}(n).$ Also notice that, in the monotonicity formula \eqref{eq:2.17}, it is required that $\lambda \in \overline{\Gamma_k(n)}$ rather than $\lambda \in \overline{\Gamma_{k-1}(n)}$ as Lemma \ref{lemma-3th} below.

\begin{lemma}\label{lemma-3th}.
Suppose that $\lambda\in \Gamma_{k-1}(n)$, $2\leq k\leq n-1$. If $\lambda_1\geq \lambda_2\geq\ldots\geq \lambda_n$, then
$$
\sigma_{k-1;1}(\lambda)\geq 0
$$
is equivalent to
$$
0\leq \sigma_{k-1;1}(\lambda)\leq \sigma_{k-1;2}(\lambda)\leq\ldots\leq \sigma_{k-1;n}(\lambda).
$$
\end{lemma}

\begin{proof}
We claim that $\sigma_{k-1;2}(\lambda)\geq \sigma_{k-1;1}(\lambda)$, that is
$$\sigma_{k-1}(\lambda_1,\lambda_3,\ldots,\lambda_n)\geq
\sigma_{k-1}(\lambda_2,\lambda_3,\ldots,\lambda_n).$$
If the claim is true, by using it again, we obtain
$$\sigma_{k-1;n}(\lambda)\geq \ldots\sigma_{k-1;2}(\lambda)\geq \sigma_{k-1;1}(\lambda).$$
Since $\lambda \in \Gamma_{k-1}(n),$ by \eqref{eq:2.12} we have
$$\sigma_{l;1}(\lambda)>0,\indent l+1\leq k-1,$$
which, together with the assumption
$$\sigma_{k-1}(\lambda_2,\lambda_3,\ldots,\lambda_n)=\sigma_{k-1,1}(\lambda)\geq 0,$$
yields $(\lambda_2,\lambda_3,\ldots,\lambda_n)\in {\Gamma_{k-1}(n-1)}.$ Using
\eqref{eq:2.12} again, we obtain
$$\sigma_{k-2}(\lambda_3,\lambda_4,\ldots,\lambda_n)=
\sigma_{k-2,2}(\lambda_2,\lambda_3,\ldots,\lambda_n)\geq 0.$$
 Let $\varepsilon=\lambda_1-\lambda_2\geq 0$, we have
\begin{equation*}
\begin{split}
&\sigma_{k-1;2}(\lambda)=\sigma_{k-1}(\lambda_1,\lambda_3,\ldots,\lambda_n)\\
=&(\lambda_2+\varepsilon)\sigma_{k-2}(\lambda_3,\ldots,\lambda_n)
+\sigma_{k-1}(\lambda_3,\ldots,\lambda_n)
\\
\geq & \lambda_2\sigma_{k-2}(\lambda_3,\ldots,\lambda_n)
+\sigma_{k-1}(\lambda_3,\ldots,\lambda_n)\\
=&\sigma_{k-1}(\lambda_2,\lambda_3,\ldots,\lambda_n)=\sigma_{k-1;1}(\lambda),
\end{split}
\end{equation*}
thus, the claim is proved.
\end{proof}

We will give a characterization of ellipticity for the linearized operator of $S_k[\psi]=c$ with $c\in \mathbb{R}.$

\begin{theorem}\label{3th}
For any $c\in \mathbb{R}$, there exists $\lambda\in \Gamma_{k-1}(n)$ such that
\begin{equation}\label{eq:3.6}
0< \sigma_{k-1;1}(\lambda)\leq \sigma_{k-1;2}(\lambda)\leq\ldots\leq \sigma_{k-1;n}(\lambda),
\end{equation}
and $\psi=\frac{1}{2}\sum_{i=1}^n\lambda_iy_i^2$ is a $(k-1)$ - convex  solution of  $k-$Hessian
equation $S_k[\psi]=c$, moreover, the linearized operators of $S_k[u]$ at $\psi$
\begin{equation}\label{eq:3.7}
\mathcal{L}_\psi=\sum_{i=1}^n\sigma_{k-1,i}(\lambda)\partial_i^2
\end{equation}
is uniformly elliptic.
\end{theorem}

\begin{proof}
If $c>0$, we take $\lambda_1=\lambda_2=\ldots=\lambda_n=[\frac{c}{(^n_k)}]^{\frac{1}{k}}>0,$ then
$\sigma_k(\lambda)=c$; If $c=0$, see Theorem \ref{2th} (1). It is left to consider
the case $c<0.$

Notice that $\mathbf{P_2}\neq\emptyset$ when $k<n,$ so we can choose $\delta=(\delta_2,\ldots,\delta_n)\in \mathbf{P_2}\subset\partial\Gamma_{k-1}(n-1)$
with $\delta_2\geq \delta_3\geq\ldots\geq \delta_n$, that is,
$$\sigma_{k-1}(\delta_2,\ldots,\delta_n)=0,\sigma_k(\delta_2,\ldots,\delta_n)<0$$
Obviously $\delta_2>0.$ Choosing $1>t>0$ small such that, for $\delta_1=\delta_{2}+1$,
\begin{equation*}
\begin{split}
&\sigma_{k-1}(\delta_2+t,\ldots,\delta_n+t)>0,\\
&\delta_1\sigma_{k-1}(\delta_2+t,\ldots,\delta_n+t)+\sigma_k(\delta_2+t,\ldots,\delta_n+t)<0.
\end{split}
\end{equation*}
Then $$\sigma_k(\delta_1,\delta_2+t,\ldots,\delta_n+t)=\delta_1\sigma_{k-1}(\delta_2+t,\ldots,\delta_n+t)
+\sigma_k(\delta_2+t,\ldots,\delta_n+t)<0.$$
Since $\sigma_k(s\lambda)=s^k\sigma_k(\lambda)$ for $s>0,$ we can choose suitable $s>0$ such that
$$
\sigma_k(s\delta_1,s(\delta_2+t),\ldots,s(\delta_n+t))=c.
$$
Let $\lambda=(s\delta_1,s(\delta_2+t),\ldots,s(\delta_n+t))$, then
$$
\sigma_k(\lambda)=c.
$$
The fact $(\lambda_2,\lambda_3,\ldots,\lambda_n)\in \Gamma_{k-1}(n-1)$  and \eqref{eq:2.12}
lead to
$$
\sigma_l(\lambda_2,\lambda_3,\ldots,\lambda_n)>0,\indent 1\leq l\leq k-1.
$$
Therefore, by virtue of $\lambda_1=s\delta_1>0,$
$$
\sigma_1(\lambda)=\lambda_1+\sigma_1(\lambda_2,\lambda_3,\ldots,\lambda_n)>
\sigma_1(\lambda_2,\lambda_3,\ldots,\lambda_n)>0
$$
and
$$
\sigma_l(\lambda)=\lambda_1\sigma_{l-1}(\lambda_2,\lambda_3,\ldots,\lambda_n)
+\sigma_l(\lambda_2,\lambda_3,\ldots,\lambda_n)>0,\indent 2\leq l\leq k-1.
$$
By the definition of $\Gamma_{k-1}(n)$, we have proved that $\lambda\in \Gamma_{k-1}(n)$. Noticing
$$
\lambda\in \Gamma_{k-1}(n),\indent \lambda_1\geq \lambda_2\ldots\geq \lambda_n,\indent  \sigma_{k-1}
(\lambda_2,\lambda_3,\ldots,\lambda_n)>0
$$
and applying Lemma \ref{lemma-3th}, we see that
 $$
0<\sigma_{k-1;1}(\lambda)\leq \sigma_{k-1;2}(\lambda)\leq\ldots\leq \sigma_{k-1;n}(\lambda),
$$
which leads to  that the operator \eqref{eq:3.7} is uniformly elliptic. This completes the proof.
\end{proof}

\begin{theorem}\label{4th}
For any $0<c,$  there exists $\lambda\in \Gamma_{k}(n)$ such that
$$
  \left\{
\begin{array}{l}
0< \sigma_{k-1;i}(\lambda), \ \ \  1\leq i\leq n \\
 \sigma_{k+l-1}(\lambda)>0, \sigma_{k+l}(\lambda)<0, \ \ \ 1\leq l\leq n-k.
\end{array}\right.
$$
In particular, there exists $\lambda\in \Gamma_{n}(n)$ such that
$$
\sigma_k(\lambda)=c.
$$
Therefore, for $1\leq l\leq n-k$, $\psi=\frac{1}{2}\sum_{i=1}^n\lambda_iy_i^2$ is a $(k+l-1)$ - convex solution of  $k-$Hessian
equation $S_k[\psi]=c$ which is not $(k+l)-$convex. Moreover, the linearized operators of $S_k[\,\cdot\,]$ at $\psi$
$$
\mathcal{L}_\psi=\sum_{i=1}^n\sigma_{k-1,i}(\lambda)\partial_i^2
$$
is uniformly elliptic.
\end{theorem}
\begin{proof}
 In the proof of Theorem \ref{3th},  replacing $\sigma_k$ by $\sigma_{k+l}$, we obtain that $\lambda\in \Gamma_{k+l-1}(n)$ with $\sigma_{k+l}(\lambda)<0$ for $1\leq l\leq n-k$. For this $\lambda$,
choosing $s>0$ such that $\sigma_{k}(s\lambda)=c.$  The other part of proof is the same as those in Theorem \ref{3th}.
\end{proof}


\section{Existence of $C^\infty$ local Solutions}\label{section3}

In this section, by the classification of precedent section,
we now prove  Theorem \ref{th:Main1} which is stated in the following precise version.

\begin{theorem}\label{th:M1}
For $2\leq k\leq n-1,$ let $f=f(y,u,p)$ be defined and continuous near a point $Z_0=(0, 0, 0)\in \mathbb{R}^n\times\mathbb{R}\times\mathbb{R}^{n}$ and $0<\alpha<1$. Assume that $f$ is $C^\alpha$ with
respect to $y$ and  $C^{2,1}$ with respect to $u, p$. We have the following results
\begin{itemize}
 \item[(1)] If  $f(Z_0)=0$, then \eqref{eq:1.2} admits  a $(k-1)$-convex local solution $C^{2, \alpha}$ near $y_0=0$,
which is not $(k+1)$-convex and of  the following form
\begin{equation}\label{eq:7.1}
u(y)=\frac{1}{2}\sum^{n}_{i=1}\tau_{i}y^{2}_{i}+\varepsilon'\varepsilon^{4}w(\varepsilon^{-2}
y)
\end{equation}
with arbitrarily fixed $(\tau_{1},\ldots,\tau_{n})\in \mathbf{P}_{2}$, $\varepsilon'$ is defined in \eqref{eq:1.5BB} and $\varepsilon>0$ very small, the function $w$ satisfies
\begin{equation}\label{eq:7.2}\left\{
\begin{array}{l}
\|w\|_{C^{2,\alpha}}\leq 1\\
w(0)=0,\nabla w(0)=0.
\end{array}\right.
\end{equation}

\item [(2)]  If $f\geq 0$ near $Z_0$, then the equation \eqref{eq:1.2}  admits  a $k$ -convex local solution $u\in
C^{2, \alpha}$ near $y_0=0$  which  is not $(k+1)$ - convex and  of the form \eqref{eq:7.1}.
\end{itemize}
Moreover, the equation \eqref{eq:1.2} is uniformly elliptic with respect to the solution
\eqref{eq:7.1}. If $f\in C^\infty$ near $Z_0$, then $u\in C^\infty$ near $y_0$.
\end{theorem}

Theorem \ref{th:M1} is exactly the part (1) and (2) of Theorem \ref{th:Main1}.

We now proceed to take the change of unknown function  $u\leftrightarrow w$ and the change of variable $y\leftrightarrow x$, the aim is to consider the equation \eqref{eq:1.2} in the domain $B_1(0)$ with the new variable $x$ and then the so-called "local" enter into the new equation itself, see \eqref{eq:7.3}-\eqref{eq:7.4}. The trick is from Lin \cite{Lin}. Let $\tau=(\tau_1,\ldots,\tau_n)\in\mathbf{P}_{2}$, then
$\psi(y)=\frac{1}{2}\sum_{i=1}^n\tau_i y_i^2$ is a polynomial-type solution
of
$$
S_k[\psi]=0.
$$
We follow Lin \cite{Lin} to introduce the following function
$$
u(y)=\frac{1}{2}\sum_{i=1}^n\tau_i
y_i^2+\varepsilon'\varepsilon^{4}w(\varepsilon^{-2}y)
=\psi(y)+\varepsilon'\varepsilon^{4}w(\varepsilon^{-2}y),\indent
\tau\in\mathbf{P}_{2},\,\,\varepsilon>0,
$$
as a candidate of solution for equation \eqref{eq:1.1}. Noting $y=\varepsilon^2 x,$
we have
$$
(D_{y_j} u)(x)=\tau_j \varepsilon^2 x_j+\varepsilon'\varepsilon^2 w_j(x),\indent j=1, \cdots, n,
$$
and
$$
(D_{y_jy_k} u)(x)=\delta^j_k \tau_j +\varepsilon' w_{jk}(x),\indent
j, k=1, \cdots, n,
$$
where $\delta^j_k$ is the Kronecker symbol, $w_j(x)=(D_{y_j} w)(x)$ and $w_{jk}(x)=(D^2_{y_{jk}} w)(x)$. Then \eqref{eq:1.2} transfers  to
$$
\tilde{S}_{k}(w)=\tilde{f}_\varepsilon(x,w(x),Dw(x)),\indent x\in B_1(0)=\{x\in
\mathbb{R}^{n}; |x|< 1\},
$$
where
$$
\tilde{S}_{k}[w]=S_{k}(\delta_{i}^{j}\tau_i+\varepsilon'
w_{ij}(x))=S_{k}(r(w)),
$$
with symmetric matrix $r(w)=(\delta_{i}^{j}\tau_i+\varepsilon' w_{ij}(x))$, and
$$
\tilde{f}_\varepsilon(x,w(x),Dw(x))=f(\varepsilon^{2}x,
\varepsilon^{4}\psi(x)+\varepsilon'\varepsilon^{4}w(x),\tau_1 \varepsilon^2 x_1+\varepsilon'\varepsilon^2 w_1(x),
\cdots, \tau_n \varepsilon^2 x_n+\varepsilon'\varepsilon^2 w_n(x)).
$$

 We now explain the smooth condition on $f=f(x, u, p)$ and its norms, that is,
$f$ is H\"{o}lder continuous with respect to $x$, denoted by $f\in C^\alpha_x$, and $C^{2,1}$ with respect to $u,p$, denoted by $C^{2,1}_{u,p}$. We will consider $f=f(x,u,p)$ defined on
$$
\mathcal{B}=\left\{(x,u,p)\in R^n\times R \times R^n:x\in B_1(0), |u|\leq A, |p|\leq A \right\}
$$
for some fixed $A>0$. We say $f\in C^\alpha_x$ if
$$
\|f\|_{C^{\alpha}_{x}}=\|f\|_{L^\infty(\mathcal{B})}+
\sup_{(x, u, p)\in \mathcal{B},  (z, u, p)\in \mathcal{B},x\neq z} \frac{|f(x,u,p)-f(z,u,p)|}{|x-z|^\alpha} <\infty.
$$
When $f=f(x)$, then $f\in C^\alpha_x$ is the usual $f\in C^\alpha(B_1(0)).$ When defining $f=f(z,u,p)\in C^{2,1}_{u,p}$, we regard $z$ as a parameter as follow:
 $$
\|f\|_{C^{1,1}_{u,p}}=\sup_{\mathcal{B}}\left\{|D^{\beta}_{u,p}f(z,u,p)|:0\leq|\beta|\leq2\right\}<\infty
$$
and
$$
\|f\|_{C^{2,\alpha}_{u,p}}=\|f\|_{C^{1,1}_{u,p}}+
\sup_{(z, u, p)\in \mathcal{B},  (z, u', p')\in \mathcal{B},(u,p)\neq (u',p')}\left\{\frac{|D^{\beta}_{u,p} f(z,u,p)-
D^{\beta}_{u,p} f(z,u',p')|}{|(u,p)-(u',p')|^\alpha},|\beta|=2\right\}<\infty,
$$
for $0<\alpha< 1$ and a similar definition for $\|f\|_{C^{2, 1}_{u,p}}$.  Here and later on, without confusion, we will denote $\|f\|_{C^{\alpha}_{x}}, \|f\|_{C^{1,1}_{u,p}}$  and $\|f\|_{C^{2, 1}_{u,p}}$ as $\|f\|_{C^{\alpha}}, \|f\|_{C^{1,1}}$ and $\|f\|_{C^{2, 1}}$ respectively .

Similar to \cite{Lin} we consider the nonlinear operators
\begin{equation}\label{eq:7.3}
G(w)=\frac{1}{\varepsilon'}[{S}_{k}(r(w))-\tilde{f_\varepsilon}(x, w,
Dw)],\qquad \mbox{in}\,\,\, B_1(0).
\end{equation}
The linearized operator of $G$ at $w$ is
\begin{equation}\label{eq:7.4}
L_{G}(w)=\sum_{i,j=1}^{n}\frac{\partial S_{k}(r(w))}{\partial
r_{ij}}\partial^2_{ij}+\sum_{i=1}^{n}a_{i}\partial_{i}+a,
\end{equation}
where
$$
a_{i}=-\frac{1}{\varepsilon'}\frac{\partial
\widetilde{f_\varepsilon}(x,z,p_{i})}{\partial
p_{i}}(x,w,Dw)=-\varepsilon^2\frac{\partial f}{\partial p_{i}}
$$
$$
a=-\frac{1}{\varepsilon'}\frac{\partial
\widetilde{f_\varepsilon}(x,z,p_{i})}{\partial
z}(x,w,Dw)=-\varepsilon^{4}\frac{\partial f}{\partial z}.
$$
Hereafter, we denote $S_{k}^{ij}(r(w))=\frac{\partial
S_{k}(r(w))}{\partial r_{ij}}$.

\begin{lemma}\label{lm:4.2}
Assume that $\tau\in\mathbf{P}_{2}$ and $\|w\|_{C^{2}(B_1(0))}\leq 1$,
 then the operator $L_{G}(w)$ is a uniformly elliptic operator if $\varepsilon>0$ is small enough.
\end{lemma}

\begin{proof}
In order to prove the uniform ellipticity of
$L_G(w)$
$$
\sum_{i,j=1}^nS_{k}^{ij}(r(w)\xi_i\xi_j\geq c|\xi|^2,\indent \forall (x,\xi)\in B_1(0)\times \mathbb{R}^n,
$$
 it suffices to prove
\begin{equation}\label{eq:7.5}\left\{
\begin{array}{l}
S_{k}^{ii}(r(w))=\sigma_{k-1;i}(\tau_1,\tau_2,\ldots,\tau_n)+O(\varepsilon'), \indent 1\leq i\leq n\\
S_{k}^{ij}(r(w))=O(\varepsilon'), \indent i\neq j,
\end{array}\right.
\end{equation}
because if it does hold, we see by \eqref{eq:3.6} that
 $$
S_{k}^{ii}(r(w))-\sum_{j=1,j\neq i}|S_{k}^{ij}(r(w))|>
 \frac{1}{2}\sigma_{k-1,i}(\tau_1,\ldots,\tau_n)>0,\indent 1\leq i\leq n
$$
if $\varepsilon>0$ is small enough, then  the matrix $(S_{k}^{ij}(r(w)))$
 is strictly diagonally dominant  and  $L_{G}(w)$ is a uniformly elliptic operator.

Indeed, Since $S_k(r)$ is the sum all principal minors of order $k$
of the Hessian $\det(r)$ , then
$$
S_{k}^{ll}(r(w))=S_{k-1}(r(w;l,l))
$$
where $r(w;l,l)$ is a $(n-1)\times(n-1)$ matrix determined from $r$ by deleting the
$l-th$ row and $l-th$ column. But $r(w)=(\delta_{i}^{j}\tau_i+\varepsilon w_{ij}(x))$, then
$$
S_{k-1}(r(w;l,l))=\sigma_{k-1;l}(\tau_1,\tau_2,\ldots,\tau_n)+O(\varepsilon')
$$
and
$$
S_{k}^{ij}(r(w))=O(\varepsilon'),\indent i\neq j.
$$
Proof is done.
\end{proof}

We follows now the idea of Hong and Zuily \cite{HZ} to prove the
existence and a priori estimates of solution for linearized
operator.  In fact, if  following the proof of \cite{HZ} step by step, we can also obtain
the existence of the local solution if $f$ is smooth enough, the reason is that $L_{G}(w)$ being
uniformly elliptic  can regarded as a special case of  $L_{G}(w)$ being degenerately elliptic
in \cite{HZ}. But if $f\in C^\alpha_x,$ which is the least requirement in classical Schauder estimates,
their proof \cite{HZ} does not work anymore because the degeneracy results in the loss of regularity. 
In our case, although $L_{G}(w)$ is uniformly elliptic,
the existence and the priori Schauder estimates of classical solutions
can not be directly obtained. The difficulty lies in that we do not know whether the
coefficient $a$ of the term $a\, u$ in \eqref{eq:7.4} is non-positive. After
proving the existence of the linearized equation (Lemma \ref{lemma:schauder}),  we can  employ Nash-Moser
procedure to prove the existence of local solution for
\eqref{eq:1.2} in H\"{o}lder space.  We shall use the following schema :
\begin{equation}\left\{ \label{eq:7.6}
\begin{array}{l}
w_{0}=0,\indent w_{m}=w_{m-1}+\rho_{m-1},\,\, m\ge 1,\\
L_{G}(w_{m})\rho_{m}=g_{m},\text {in}\ \ B_1(0),\\
\rho_{m}=0\indent \text{on}\ \  \partial B_1(0),\\
g_{m}=-G(w_{m})\, ,\\
\end{array}\right.
\end{equation}
where
$$
g_0(x)=\frac{1}{\varepsilon'}\Big(\sigma_k(\tau)-f\big(\varepsilon^2 x, \varepsilon^4
\psi(x), \varepsilon^2(\tau_1 x_1,\tau_2 x_2,\ldots,\tau_n x_n)\big)\Big)\, .
$$

It is pointed out on page 107, \cite{GT} that, if the operator
$L_{G}$ does not satisfy the condition $a\leq 0,$ as is well known
from simple examples, the Dirichlet problem for $L_{G}(w)\rho=g$ no
longer has a solution in general. Notice $a$ in \eqref{eq:7.4} has
the factor $\varepsilon^4$, we will take advantage of the smallness of
$a$ to obtain the uniqueness and existence of solution for Dirichlet
problem \eqref{eq:7.7}. We will assume $\|w_{m}\|_{C^{2,\alpha}}\leq A$ rather than
$\|w_{m}\|_{C^{2,\alpha}}\leq 1$ as in \cite{HZ}, the advantage
 is to see how the procedure depends on $A$. Actually $A$ can be taken as 1. We have uniformly Schauder estimates of its solution as follows.

\begin{lemma}\label{lemma:schauder}
Assume that $\|w\|_{C^{2, \alpha}(B_1(0))}\le A$. Then there exists
a unique solution  $\rho\in C^{2, \alpha}(\overline{B_1(0)})$ to the
following Dirichlet problem
\begin{equation}\label{eq:7.7}\left\{
\begin{array}{l}
L_{G}(w)\rho=g,\quad \text {in}\ \ B_1(0),\\
\rho=0\indent \text{on}\ \  \partial B_1(0)\,
\end{array}\right.
\end{equation}
for all $g\in  C^{\alpha}(\overline{B_1(0)}).$ Moreover,
\begin{equation}\label{eq:7.8}
\|\rho\|_{C^{2, \alpha}(\overline{B_1(0)})}\le C\|g\|_{C^{
\alpha}(\overline{B_1(0)})}, \quad \forall g\in C^{
\alpha}(\overline{B_1(0)}),
\end{equation}
where the constant $C$ depends on $A, \tau$ and
$\|f\|_{C^{2,1}}$. Moreover, $C$ is independent of    $0<\varepsilon\le \varepsilon_0$
 for some $\varepsilon_0>0$.
\end{lemma}

By the virtue of \eqref{eq:7.4}, we write \eqref{eq:7.7} as
\begin{equation}\label {eq:7.9}\left\{
\begin{array}{l}
L_{G}(w)\rho=\sum_{i,j=1}^{n}\frac{\partial S_{k}(r(w))}{\partial
r_{ij}}\partial_i\partial_j\rho+\sum_{i=1}^{n}a_{i}\partial_{i}\rho+a\rho=g,\quad
\text {in}\ \ B_1(0),\\
 \rho=0\indent \text{on}\ \  \partial B_1(0)\,
\end{array}\right.
\end{equation}
where
$$
a_{i}=-\varepsilon^2\frac{\partial  f}{\partial p_{i}},\quad
a=-\varepsilon^{4}\frac{\partial f}{\partial z}.
$$
Noticing that for the functions, such as  $\frac{\partial S_{k}(r(w))}{\partial r_{ij}}$,
$a_i=a_i(x,w(x),Dw(x))$, $a=a(x,w(x),Dw(x))$ and
$g_m=-G(w_m)=g_m(x,w_m(x),Dw_m(x),D^2w_m(x))$    by \eqref{eq:7.6}, we
regard them as the functions with variable $x$. In a word, we regard
that all of the coefficients and non-homogeneous term in
\eqref{eq:7.9} are functions of variable $x.$ For example,
\begin{equation*}
\begin{split}
&\tilde{f}_\varepsilon(x,w(x),Dw(x))\\
=&f(\varepsilon^{2}x,
\varepsilon^{4}\psi(x)+\varepsilon'\varepsilon^{4}w(x),\tau_1 \varepsilon^2
x_1+\varepsilon'\varepsilon^2 w_1(x), \cdots, \tau_n \varepsilon^2
x_n+\varepsilon'\varepsilon^2 w_n(x)).
\end{split}
\end{equation*}

\begin{proof}[Proof of Lemma \ref{lemma:schauder}] Let
 $$
\mu(\tau)=\inf\left\{ \sum_{i,j=1}^nS_{k}^{ij}(r(w)\xi_i\xi_j,
 \indent \forall x\in B_1(0), |\xi|=1, \|w\|_{C^{2,
\alpha}(B_1(0))}\le A\right\}.
$$
By Lemma \ref{lm:4.2}, $\mu(\tau)>0$. Applying Theorem 3.7 in \cite{GT} to the solution $u\in
C^0(\overline{B_1(0)})\cap C^2(B_1(0))$ of
$$
\left\{
\begin{array}{l}
L_{G}(w)u=\sum_{i,j=1}^{n}\frac{\partial S_{k}(r(w))}{\partial
r_{ij}}\partial_i\partial_ju+\sum_{i=1}^{n}a_{i}\partial_{i}u=g,\quad
\text {in}\ \ B_1(0),\\
 u=0,\indent \text{on}\ \  \partial B_1(0),\,
\end{array}\right.
$$
we have
\begin{equation}\label{eq:7.10}
\sup |u|\leq \frac{C}{\mu(\tau)}\|g\|_{C^0(\overline{B_1(0)})},
\end{equation}
where $C=e^{2(\beta+1)}-1$ and
$\beta=\sup\left\{\frac{|a_i|}{\mu(\tau)}: i=1,2,\ldots,n\right\}$.

Let $C_1=1-C\sup\frac{|a|}{\mu(\tau)}$ with $C$ being the constant
in \eqref{eq:7.10}. If we choose $\varepsilon_0>0$ small (
since $a=O(\varepsilon^4)$ is small), then $C_1>\frac{1}{2}$ is independent of
$0<\varepsilon<\varepsilon_0.$ By applying Corollary 3.8 in \cite{GT} to
the solution $\rho$ to Dirichlet problem \eqref{eq:7.9}, we have
\begin{equation}\label{eq:7.11}
\sup |\rho|\leq \frac{1}{C_1}\left[\sup_{\partial
B_1(0)}|\rho|+\frac{C}{\mu(\tau)}\|g\|_{C^0(\overline{B_1(0)})}\right]=
\frac{C}{C_1\mu(\tau)}\|g\|_{C^0(\overline{B_1(0)})}.
\end{equation}
From \eqref{eq:7.11} we see that the homogeneous problem
$$
\left\{
\begin{array}{l}
L_{G}(w)\rho=\sum_{i,j=1}^{n}\frac{\partial S_{k}(r(w))}{\partial
r_{ij}}\partial_i\partial_j\rho+\sum_{i=1}^{n}a_{i}\partial_{i}\rho+a\rho=0,\quad
\text {in}\ \ B_1(0),\\
 \rho=0\indent \text{on}\ \  \partial B_1(0)\,
\end{array}\right.
$$
only possesses the trivial solution. Then we can apply a Fredholm
alternative, Theorem 6.15 in \cite{GT}, to the inhomogeneous problem
\eqref{eq:7.9} for which we can assert that it has a unique $C^{2,
\alpha}(\overline{B_1(0)})$ solution for all $g\in
C^{\alpha}(\overline{B_1(0)}).$

With the existence and uniqueness at hand, we can apply Theorem 6.19
 \cite{GT} to obtain higher regularity up to boundary for solution
 to \eqref{eq:7.9}. Besides this, we have the  Schauder estimates  (see
  Problem 6.2 , \cite{GT})
\begin{equation}\label{eq:7.12}
\|\rho\|_{C^{2,\alpha}}\leq
C(A,\tau,\|f\|_{C^{1,1}})\left[\|\rho_{k}\|_{C^0(\overline{B_1(0)})}+
\|g_{k}\|_{C^{\alpha}(\overline{B_1(0)})}\right],
\end{equation}
where $C$ depends on $C^{\alpha}-$norm of all of the
coefficients; the uniform ellipticity; boundary value and boundary
itself. Now we explain the dependence of
$C(A,\tau,\|f\|_{C^{1,1}}).$ Firstly, since the first two
derivatives of $w$ have come  into the principal coefficients
$\frac{\partial S_{2}(r(w))}{\partial r_{ij}}$, then their
$C^{\alpha}$-norms must be involved in $\|w\|_{C^{2,\alpha}}$. That is, $\|w\|_{C^{2,\alpha}}\leq A$  arise into $C$. Similarly,
by the virtue of the coefficients $a_i$ and $a$,  we have that $\|f\|_{C^{1,1}}$
and $\|w\|_{C^{2,\alpha}}\leq A$ must arise into $C$.
  Secondly, it depends on the uniform ellipticity, that is,
$$
\inf\left\{ \sum_{i,j=1}^nS_{k}^{ij}(r(w)\xi_i\xi_j,
 \indent \forall x\in B_1(0), |\xi|=1, \|w\|_{C^{2,
\alpha}(B_1(0))}\le A\right\}
$$
and
$$
\sup\left\{ \sum_{i,j=1}^nS_{k}^{ij}(r(w)\xi_i\xi_j,
 \indent \forall x\in B_1(0), |\xi|=1, \|w\|_{C^{2,
\alpha}(B_1(0))}\le A\right\},
$$
so ($\tau=(\tau_1,\tau_2,\ldots,\tau_n)$) and $A$ arise into $C$.

Thirdly, Since its boundary value is zero and boundary $\partial B_1(0)$
is $C^\infty$, the these two ingredients do not occur into $C$.
  Substituting \eqref{eq:7.11} into \eqref{eq:7.12}, we obtain
  \eqref{eq:7.8}.
\end{proof}

It follows from the standard elliptic theory (see Theorem 6.17 in \cite{GT} and Remark 2 in \cite{CNS1}) and an iteration argument that we obtain.

\begin{corollary}\label{corollary1}
Assume that $u\in C^{2, \alpha}(\Omega)$ is a solution of \eqref{eq:1.2}, and the linearized operators with respect to $u$,
$$
\mathcal{L}_u=\sum_{i,j=1}^{n}\frac{\partial
S_{k}(u_{ij})}{\partial r_{ij}}\partial^2_{ij}
-\sum_{i=1}^{n}\frac{\partial f}{\partial
p_i}(y,u(y),Du(y))\partial_{i}-\frac{\partial f}{\partial
z}(y,u(y),Du(y)),
$$
is uniformly elliptic. If $f\in C^\infty$, then $u\in C^\infty(\Omega)$.
\end{corollary}

Using Lemma \ref{lemma:schauder} above, we can use the procedure \eqref{eq:7.6} to
construct the sequence $\{w_m\}_{m\in \mathbb{N}}$. Now we
study the convergence of $\{w_m\}_{m\in \mathbb{N}}$ and $\{g_m\}_{m\in \mathbb{N}}$.

\begin{proposition}\label{nms}
Let $\{w_m\}_{m\in \mathbb{N}}$ and $\{g_m\}_{m\in \mathbb{N}}$ be the
sequence in \eqref{eq:7.6}. Suppose that
$\|w_{j}\|_{C^{2,\alpha}}\leq A$  for $j=1,2,\ldots,l$. Then we have
\begin{equation}\label{eq:7.13}
\|g_{l+1}\|_{C^{\alpha}}\leq
C\|g_{l}\|^{2}_{C^{\alpha}} ,
\end{equation}
where $C$ is some positive constant depends only on $\tau$, $A$ and
 $ \|{f}\|_{C^{2,1}}.$
 In particular, $C$ is independent of $l.$
\end{proposition}

\begin{proof}
By applying Taylor expansion with integral-typed remainder to
\eqref{eq:7.3}, we have
\begin{align*}
-g_{l+1}&=G(w_{l}+\rho_{l})=G(w_{l})+L_{G}(w_{l})\rho_{l}+Q(w_{l},\rho_{l})\\
&=-g_l+L_{G}(w_{l})\rho_{l}+Q(w_{l},\rho_{l})=Q(w_{l},\rho_{l}),
\end{align*}
where $Q$ is the quadratic error of $G$ which consists of $S_{k}$ and $f$,
\begin{equation}\label{eq:7.15}
\begin{split}
Q(w_{l},\rho_{l})=&\sum_{ij,st}\frac{1}{\varepsilon}\int
(1-\mu)\frac{\partial^{2}S_{k}(w_l+\mu\rho_l)}{\partial
w_{ij}\partial w_{st}}d\mu\,(\rho_{l})_{ij}\,(\rho_{l})_{st}\\
&-\sum_{i,j}\frac{1}{\varepsilon}\int
(1-\mu)\frac{\partial^{2}\tilde
f_\varepsilon(w_l+\mu\rho_l)}{\partial
w_{i}\partial w_{j}}d\mu\,(\rho_{l})_{i}\,(\rho_{l})_{j}\\
&-\frac{1}{\varepsilon}\sum_{i}\int (1-\mu)\frac{\partial^{2}\tilde
f_\varepsilon(w_l+\mu\rho_l)}{\partial
w\partial w_{i}}d\mu(\rho_{l})_{i}(\rho_l)\\
&-\frac{1}{\varepsilon}\int(1-\mu)\frac{\partial^{2}\tilde
f_\varepsilon(w_l+\mu\rho_l)}{\partial w^{2}}d\mu\cdot
\rho^{2}_l\\
&=I_1+I_2+I_3+I_4
\end{split}
\end{equation}

Since $S_k((r(w)))$ is a $k$-order homogeneous polynomial with
variable $r_{ij}(r(w))$ and $\tilde{f_\varepsilon}(x,w,Dw)$ is independent
of $r_{ij},$ we see that
 \begin{equation*}
\begin{split}
\left|\frac{\partial^{2}S_{k}(w_l+\mu\rho_l)}{\partial
w_{ij}\partial w_{st}}\right|&=\varepsilon'^2\left|\frac{\partial^{2}S_{k}}{\partial
w_{ij}\partial{w_{st}}}(\delta^{j}_{i}\tau_{i}+\varepsilon'
(w_l+\mu\rho_l)_{ij})\right|\\
&=\varepsilon'^2\sum_{j=2}^kC(j,\tau)[\partial^2(w_l+\mu\rho_l)]^{j-2}
\end{split}
\end{equation*}

\begin{equation*}
\begin{split}
\left|\frac{\partial^{2}\tilde f_\varepsilon(w_l+\mu\rho_l)}{\partial
w_{i}\partial w_{j}}\right|&=\left|\frac{\partial^{2}[f(\varepsilon
x,\varepsilon^{4}\psi+\varepsilon'\varepsilon^{4}(w_l+\mu\rho_l),\varepsilon^{2}D\psi+
\varepsilon'\varepsilon^{2}D(w_l+\mu\rho_l))]}{\partial
w_{i}\partial w_{j}}\right|\\
&\leq \varepsilon'^2\varepsilon^{4}\cdot
\|f\|_{C^{1,1}},
\end{split}
\end{equation*}
\begin{equation*}
\begin{split}
\left|\frac{\partial^{2}\tilde f_\varepsilon(w_l+\mu\rho_l)}{\partial
w\partial w_{i}}\right|&=\left|\frac{\partial^{2}[f(\varepsilon
x,\varepsilon^{4}\psi+\varepsilon'\varepsilon^{4}(w_l+\mu\rho_l),\varepsilon^{2}
D\psi+\varepsilon'\varepsilon^{2}D(w_l+\mu\rho_l))]}{\partial
w\partial w_{i}}\right|\\
&\leq \varepsilon'^2\varepsilon^{6}\|f\|_{C^{1,1}},
\end{split}
\end{equation*}
\begin{equation*}
\begin{split}
\left|\frac{\partial^{2}\tilde f_\varepsilon(w_l+\mu\rho_l)}{\partial w^{2}}\right|&=
\left|\frac{\partial^{2}[f(\varepsilon
x,\varepsilon^{4}\psi+\varepsilon'\varepsilon^{4}(w_l+\mu\rho_l),\varepsilon^{2}
D\psi+\varepsilon'\varepsilon^{2}D(w_l+\mu\rho_l))]}{\partial
w^{2}}\right|\\
&=\varepsilon'^2\varepsilon^{8}\|f\|_{C^{1,1}}.
\end{split}
\end{equation*}
Thus, $I_i (1\leq i\leq 4)$ in $Q$ are under control by $O(\varepsilon')$,
$O(\varepsilon'\varepsilon^{4})$, $O(\varepsilon'\varepsilon^{6})$ and $O(\varepsilon'\varepsilon^{8})$
respectively. Therefore
$$
\|I_1\|_{C^{\alpha}}\leq C\sum_{j=1}^{k-1}\|\rho_{l}\|_{C^2}^j\|\rho_{l}\|_{C^{2,\alpha}}
$$
and
\begin{equation*}
\begin{split}
\|I_2\|_{C^{\alpha}}
\leq&C\|f\|_{C^{\alpha}}(\|w_l\|_{C^{1,\alpha}}+\|\rho_l\|_{C^{1,\alpha}})\|\rho_l\|^2_{C^{1}}+
C\|f\|_{C^{1,1}}\|\rho_l\|_{C^{2,\alpha}}\|\rho_l\|_{C^{1}} \\
\leq &
C\|\rho_l\|_{C^{2,\alpha}}\|\rho_l\|^2_{C^{1}}+C\|\rho_l\|^2_{C^{1}}
+C\|\rho_l\|_{C^{2,\alpha}}\|\rho_l\|_{C^{1}}
\end{split}
\end{equation*}
holds, where $C$ depends on $A$ and $\|f\|_{C^{2, \alpha}}$.
 And $\|I_3\|_{C^{\alpha}}$  and $\|I_4\|_{C^{\alpha}}$ can be
estimated similarly.   Accordingly,
\begin{equation*}
\begin{split}
\|g_{l+1}\|_{C^{\alpha}}&= \|Q(w_{l},\rho_{l})\|_{C^{\alpha}}
\leq \sum_{i=1}^4\|I_i\|_{C^{\alpha}}\\
\leq &
C\sum_{j=1}^{k-1}\|\rho_{l}\|_{C^2}^j\|\rho_{l}\|_{C^{2,\alpha}}+
C\|\rho_l\|_{C^{2,\alpha}}\|\rho_l\|^2_{C^{1}}+
\|\rho_l\|^2_{C^{1}}+C\|\rho_l\|_{C^{2,\alpha}}\|\rho_l\|_{C^{1}},
\end{split}
\end{equation*}
where $C$ is independent of $l$ but dependent of $A$ and
$\|f\|_{C^{2,1}}$. Thus, by the interpolation inequalities,  we
have
$$
\|g_{l+1}\|_{C^{\alpha}}\leq  C\sum_{j=1}^{k-1}\|\rho_{l}\|^{j+1}_{C^{2,\alpha}}+
C\|\rho_{l}\|^3_{C^{2,\alpha}},
$$
where $C$ is independent of $l$. By the Schauder estimates of Lemma
\ref{lemma:schauder}, we have
$$
\|\rho_{l}\|_{C^{2,\alpha}}\leq
C \|g_{l}\|_{C^{\alpha}}.
$$
Recall that $\|g_{l}\|_{C^{\alpha}}=\|G(w_m)\|_{C^{\alpha}}\leq C(A)$ holds provided $\|w_{j}\|_{C^{2,\alpha}}\leq A$  for $j=1,2,\ldots,l$. Combining the two estimates above, we obtain \eqref{eq:7.13}.

For the special case $f=f(y)=f(\varepsilon^2x)$, then $I_2=I_3=I_3=0$ in \eqref{eq:7.15}. So the condition $f=f(y)\in C^\alpha$ is enough for  the estimate
\eqref{eq:7.20}-\eqref{eq:7.16} below . Proof is done.
\end{proof}
Since $C$ is independent of $l$, more exactly, $A$, $\tau$ and
$\|f\|_{C^{2,1}}$  are independent of $l$. So hereafter, we can assume $A=1.$

\begin{proof}[\bf Proof of Theorem \ref{th:M1}] Set
\begin{equation}\label{eq:7.14}
d_{l+1}=C\|g_{l+1}\|_{C^{2,\alpha}}, \ \  l=0,1,2,\ldots..
\end{equation}
By \eqref{eq:7.13} and setting $C\geq 1$ we have
$$
d_{l+1}\leq d_{l}^{2}.
$$
Take $\tau\in\mathbf{P}_{2}$  such that  $\sigma_k(\tau)=f(0, 0, 0)$,
we have
\begin{align*}
& g_0(x)=-G(0)=\frac{1}{\varepsilon'}[S_k(r(0))-\tilde{f}(x, 0, 0)]\\
=&\frac{1}{\varepsilon'}\left[\sigma_k(\tau)-
f\big(\varepsilon^2 x, \varepsilon^4
\psi(x), \varepsilon^2(\tau_1 x_1, \ldots,\tau_n x_n)\big)\right]\\
=&\frac{1}{\varepsilon'}\left[(\sigma_k(\tau)-
f\big(0,0,
0)\right]+\frac{1}{\varepsilon}\left[f(0, 0, 0)-
f(\varepsilon^2 x, 0, 0)\right]\\
&+\frac{1}{\varepsilon'}\left[f(\varepsilon^2 x, 0, 0)-
f\big(\varepsilon^2 x, \varepsilon^4
\psi(x), \varepsilon^2(\tau_1 x_1, \ldots,\tau_n x_n)\big)\right]\\
=&\frac{1}{\varepsilon'}\left[f(0, 0, 0)-
f(\varepsilon^2 x, 0, 0)\right]\\
-&\frac{\varepsilon^4}{\varepsilon'} \int^1_0
\psi(x)(\partial_zf)\Big(\varepsilon^2 x, t\varepsilon^4
\psi(x), t\varepsilon^2(\tau_1 x_1,\ldots,\tau_n x_n)\Big) dt\\
-&\frac{\varepsilon^2}{\varepsilon'} \int^1_0 (\tau_1 x_1,\ldots,\tau_n x_n)\,\cdot\,(\partial_pf)\Big(\varepsilon^2 x, t\varepsilon^4
\psi(x), t\varepsilon^2(\tau_1 x_1,\ldots,\tau_n x_n)\Big) dt,
\end{align*}
where $\sigma_k(\tau)-f\big(0,0,0)=0$ is used. Noticing
$$\frac{1}{\varepsilon'}\|f(0, 0, 0)-
f(\varepsilon^2 x, 0, 0)\|_{C^{ \alpha}(B_1(0))}\leq C\frac{\varepsilon^{2\alpha}}{\varepsilon'}
\|f(\cdot,0,0)\|_{C^\alpha(B_{\varepsilon^2}(0))},
$$
we obtain
\begin{equation}\label{eq:7.20}
 \|g_0\|_{C^{ \alpha}(B_1(0))}\le C_1\frac{\varepsilon^{2\alpha}}{\varepsilon'}  \|f\|_{C^{1,1}}.
\end{equation}
Using the definition of $\varepsilon'$ in \eqref{eq:1.5BB}, we can choose $0<\varepsilon\leq \varepsilon_0$ so small that
\begin{equation}\label{eq:7.16}
C\|g_{0}\|_{C^{\alpha}(B_1(0))}\leq \frac{1}{4}, \indent 0<\varepsilon\leq
\varepsilon_0.
\end{equation}

Notice $\varepsilon_0$ is independent of $l.$ Since
$d_{0}=C\|g_{0}\|_{C^{\alpha}}$, we have $d_1\leq d_0^2$. Then, by an
induction,

 $$
d_{l+1}\leq
2^{2^{l+1}}d_0^{2^{l+1}}\leq(2C)^{2^{l+1}}\|g_{0}\|^{2^{l+1}}_{C^{\alpha}}.
$$

Thus, by \eqref{eq:7.14} and \eqref{eq:7.16}
\begin{equation}\label{eq:7.17}
\|g_{l+1}\|_{C^{\alpha}}\leq (2C)^{2^{l+1}-1}\|g_{0}\|^{2^{l+1}}_{C^{\alpha}}\leq \left(\frac 12\right)^{2^l}\,\,\rightarrow\,\, 0,.
\end{equation}
Firstly, we claim that there exists a constant $\varepsilon_0>0,$
depending on $\tau$ and $\|f\|_{C^{2,1}}$ such that, uniformly for $\varepsilon\in (0,\varepsilon_0],$
$$
\|w_{l}\|_{C^{2,\alpha}(B_1(0))}\leq 1,\ \ \forall l\geq 1.
$$
Indeed, set $w_0=0,$  we have by \eqref{eq:7.13}
\begin{align*}
\|w_{l+1}\|_{C^{2,\alpha}(B_1(0))}&=\|\sum_{i=0}^{l}\rho_{i}\|_{C^{2,\alpha}(B_1(0))}
\leq\sum_{i=0}^{l}\|\rho_{i}\|_{C^{2,\alpha}(B_1(0))}\\
&\leq \sum_{i=0}^{l}C\|g_{i}\|_{C^{\alpha}(B_1(0))}\leq
\sum_{i=0}^{l}\Big(C\|g_{0}\|_{C^{\alpha}(B_1(0))}\Big)^{2^{i}}
\end{align*}
where $C$ is defined in Lemma \ref{nms}.  Thus, for any $l$,
\begin{equation}\label{eq:7.18}
 \|w_{l+1}\|_{C^{2,\alpha}(B_1(0))}\leq
\sum_{i=0}^{\infty}\Big(C\|g_{0}\|_{C^{\alpha}(B_1(0))}\Big)^{2^{i}}\leq
\sum_{i=0}^{\infty}2^{-2^{i}}\leq 1.
\end{equation}
Then, by Azel\`{a}-Ascoli Theorem, there is a subsequence of $w_{l}$, still denoted by $w_{l}$,
such that
$$
w_{l}\rightarrow w\indent \textup{in}~C^{2}(B_1(0)),
$$
and $w\in C^{2, \alpha}(B_1(0))$.
From \eqref{eq:7.17} and  $g_{m}=-G(w_{m})$, we have
$$
G(w)=\frac{1}{\varepsilon'}[S_k(r(w))-\tilde{f}(x, w, Dw)]=0,\qquad \mbox{on}\,\,\,
B_1(0).
$$
That means to say the function
$$
u(y)=\frac{1}{2}\sum_{i=1}^n\tau_i
y_i^2+\varepsilon' \varepsilon^{4}w(\varepsilon^{-2}y)\in C^{2, \alpha}(B_{\varepsilon^2}(0)),
$$
is a solution of
$$
S_k[u]=f(y, u, Du),\qquad \mbox{on}\,\,\,B_{\varepsilon^2}(0)\, .
$$

Now if $f(0, 0, 0)=0$, we take $\tau\in\mathbf{P}_2$, then $\sigma_{k-1}(\tau)>0,
\sigma_k(\tau)=0, \sigma_{k+1}(\tau)<0$. Noticing that symmetric matrix $r(w)=(\delta_{i}^{j}\tau_i+\varepsilon' w_{ij}(x)),$ we have
$$
S_j[u]=\sigma_j(\lambda)=\sigma_j(\tau)+O(\varepsilon'),\qquad j=1, 2,\ldots, k+1.
$$
it follows that $S_j[u]>0 (1\leq j\leq k-1), S_{k+1}[u]<0$ on $B_{\varepsilon^2}(0)$ for small $\varepsilon>0$. That is, $u$ is $(k-1)-$convex but not $(k+1)-$convex. Moreover if $S_k[u]=f\geq 0$ near $Z_0$ and
$f(Z_0)=0$, we see that $u$ is $k$-convex by definition, but not $(k+1)$-convex.

If $S_k[u]=f>0$ near $Z_0$, we use Theorem \ref{4th} to take $\tau\in \mathbb{R}^n$ given in $\Gamma_{k-l+1}(n)\setminus\overline\Gamma_{k+l}(n)$ for $1\leq l\leq n-k$, then we can get the $(k+l-1)$-convex but not $(k+l)$-convex local solutions.
From \eqref{eq:7.18} and condition $Z_0=(0,0,0)$ , we obtain  \eqref{eq:7.2}.

The $C^\infty$ regularity of solution is given by Corollary \ref{corollary1}. Thus, we have proved Theorem \ref{th:M1}.
\end{proof}

We also have the following elliptic results for $f(Z_0)<0$ which is (3) in the Theorem \ref{th:Main1}.

\begin{theorem}\label{th:M2}
For $2\leq k\leq n-1,$ let $f=f(y,u,p)$ be defined and continuous near a point $Z_0=(0, 0, 0)\in \mathbb{R}^n\times\mathbb{R}\times\mathbb{R}^{n}$ and $0<\alpha<1$. Assume that $f$ is $C^\alpha$ with
respect to $y$ and  $C^{2,1}$ with respect to $u, p$. If  $f(Z_0)<0$, then \eqref{eq:1.2} admits  a $(k-1)$-convex local solution $C^{2, \alpha}$ near $y_0=0$,
which is not $k$-convex and of  the following form
\begin{equation*}
u(y)=\frac{1}{2}\sum^{n}_{i=1}\tau_{i}y^{2}_{i}+\varepsilon'\varepsilon^{5}w(\varepsilon^{-2}
y)
\end{equation*}
with $\sigma_k(\tau_1,\ldots,\tau_n)=f(0,0,0)$, $\varepsilon>0$ very small,
$\varepsilon'$ is defined in \eqref{eq:1.5BB} and $w$ satisfies \eqref{eq:7.2}.
Moreover, the equation \eqref{eq:1.2} is uniformly elliptic with respect to the solution
above.  If $f\in C^\infty$ near $Z_0$, then $u\in C^\infty$ near $y_0$.
\end{theorem}

\begin{proof}
For $f(0, 0, 0)<0$, take $\tau\in\mathbb{R}^n$ as in Theorem \ref{3th} with $c=f(0, 0, 0)<0$ such that
$$
\sigma_{k-1}(\tau)>0,\quad \sigma_k(\tau)=f(0, 0, 0)<0,
$$
and
$$
\sigma_{k-1;n}(\tau)\geq\sigma_{k-1; n-1}(\tau)\geq\ldots\sigma_{k-1, 1}(\tau)>0.
$$
Now the proof is exactly the same as that of Theorem \ref{th:M1}.
\end{proof}


\section{Appendix}\label{section4}

The following results are  essential in the proof of our main theorem, maybe it is classical, but we can't find
a simple proof, so present here as an appendix.

\subsection{The Equivalence of Three Definitions for G{\aa}rding Cone}

The G{\aa}rding cone is originated from the G{\aa}rding theory of
hyperbolic polynomials \cite{Ga}. Each hyperbolic polynomial is
essentially real. A homogeneous polynomial of $k$-order on a
$n-$dimensional real vector space $V$ is hyperbolic with respect to
a direction $a\in V$ if the  equation $P(sa+\lambda)=0$ with
one-variable $s$ has $k$ real zeros for every real $\lambda\in V.$
Using a convenient solecism, we say $P(\lambda)$ is an $a-$
hyperbolic polynomial. For an $a-$ hyperbolic polynomial
$P(\lambda)$ and $m>1,$ then by Rolle's theorem, its directional
derivative in the direction $a,$
\begin{equation}\label{eq:2.1}
(\nabla P(\lambda),a)=\frac{d}{ds}P(sa+\lambda)\mid_{s=0}
\end{equation}
is also an $a-$ hyperbolic polynomial, see Lemma 1 in \cite{Ga}. For
an $a-$ hyperbolic polynomial $P(\lambda)$ with $P(a)>0,$ its
G{\aa}rding cone is defined by
\begin{equation}\label{eq:2.2}
\mathscr{C}(a, P, n)=\left\{\lambda\in \mathbb{R}^n: P(sa+\lambda)>0,\forall s\geq 0\right\}.
\end{equation}
G{\aa}rding \cite{Ga} has proved that $
\mathscr{C}(a, P, n)=\mathscr{C}(b, P, n)$ is an open convex cone, for any $ b\in \mathscr{C}(a, P, n) $
with vertex at the origin, and
\begin{equation}\label{eq:2.4}
(\nabla P(\lambda),\mu)\geq
kP^{\frac{k-1}{k}}(\lambda)P^{\frac{1}{k}}(\mu), \ \ \ \forall
\lambda,\mu\in \mathscr{C}(a, P, n),
\end{equation}
see (11) in \cite{Ga}, which is the simplest version of the general
G{\aa}rding inequality, Theorem 5 in \cite{Ga} . We will apply the
results above to  the $k$-th elementary symmetric polynomial
$$
\sigma_{k}(\lambda)=\sum_{1\leq i_{1}<\cdots i_{k}\le
n}\lambda_{i_{1}}\cdots\lambda_{i_{k}},
$$
which is hyperbolic with respect to
$$
e=(1,1,\ldots,1)\in \mathbb{R}^n.
$$
G{\aa}rding cone is equivalently defined in the form which is easier to be verified than \eqref{eq:2.2}
\begin{equation}\label{eq:2.6}
\Gamma_{k}(n)=\{\lambda\in\mathbb{R}^{n}|\sigma_{j}(\lambda)>0,\forall j=1,2,\cdots,k\},
\end{equation}
From this definition, it follows  that
$$
\overline{\Gamma}_{n}(n)\subset\cdots\subset\overline{\Gamma}_{k}(n)
\subset\cdots\subset\overline{\Gamma}_{1}(n).
$$
Notice that Maclaurin's inequalities
\begin{equation}\label{eq:2.7}
\left[\frac{1}{(^{n}_{k})}\sigma_{k}(\lambda)
\right]^{\frac{1}{k}}\leq
\left[\frac{1}{(^{n}_{l})}\sigma_{l}(\lambda) \right]^{\frac{1}{l}}
\end{equation}
hold for $1\leq l\leq k, \lambda\in \Gamma_{k}(n)$; (see Lemma 15.12 in \cite{Li}).
Denoting
$$
\lambda+\varepsilon
=(\lambda_1+\varepsilon,\lambda_2+\varepsilon,\ldots,\lambda_n+\varepsilon)$$
and applying the formula
$$
\sigma_{k}(\lambda+\varepsilon)=
\sum_{j=0}^{k}C(j,k,n)\varepsilon^{j}\sigma_{k-j}(\lambda),\indent \
C(j,k,n)=\frac{(^{n}_{k})(^{k}_{j})}{(^{n}_{k-j})},\lambda\in
\mathbb{R}^n,\varepsilon\in \mathbb{R}.
$$
to $\lambda+\varepsilon$; (see Section 5, \cite{Ivo}), we obtain
\begin{equation}\label{eq:2.9}
\lambda+\varepsilon\in \Gamma_{k}(n),\ \ \  \forall \varepsilon>0,
\lambda\in \overline{\Gamma_{k}(n)}.
\end{equation}
For any fixed $t$-tuple
$\{i_{1},i_{2},\cdots,i_{t}\}\subset\{1,2,\cdots,n\}$, we define
\begin{equation}\label{eq:2.10}
\sigma_{k;i_{1},i_{2},\cdots,i_{t}}(\lambda)=\frac{\partial^{t}
\sigma_{k+t}(\lambda)}{\partial\lambda_{i_{1}}\cdots
\partial\lambda_{i_{t}}}.
\end{equation}
 and then  for any $i\in\{1, 2, \cdots, n\}$,
\begin{equation}\label{eq:2.11}
\sigma_{k}(\lambda)=\lambda_{i}
\sigma_{k-1;i}(\lambda)+\sigma_{k;i}(\lambda),\ \ \  \forall
\lambda\in \mathbb{R}^n.
\end{equation}

Moreover, one can verify by \eqref{eq:2.11} that
$$
\sigma_{k;i_{1},i_{2},\cdots,i_{t}}(\lambda)=\sigma_{k}
(\lambda)|_{\lambda_{i_{1}}=\cdots=\lambda_{i_{t}}=0}.
$$
 Ivochkina \cite{Ivo} introduced
a cone defined by
\begin{equation}\label{eq:2.12}
\widetilde{\Gamma}_{k}(n)=\left\{\lambda\in\mathbb{R}^{n}|
\sigma_{k-l;i_1,\ldots,i_l}>0,l=0,1,\ldots,k,\right\}
\end{equation}
and has proved that, by her own notion of "stability set of
$k-$Hessian operator ", the cones $\Gamma_{k}(n)$ and
$\widetilde{\Gamma}_{k}(n)$ coincide.
We will take the original definition \eqref{eq:2.2} as the starting
point to prove that the three definitions above are equivalent.

\begin{theorem}\label{1p}
For the $k$-th elementary symmetric polynomial $\sigma_k(\lambda)$
with $e=(1,1,\ldots,1),$ the definitions
\eqref{eq:2.2},\eqref{eq:2.6} and \eqref{eq:2.12} are equivalent,
that is,
$$
\mathscr{C}(e, \sigma_k, n)=\Gamma_{k}(n)=\widetilde{\Gamma}_{k}(n).
$$
\end{theorem}

\begin{proof} {\bf Step 1.} We will prove $\mathscr{C}(e, \sigma_k, n)= \Gamma_{k}(n). $ We first prove
$\mathscr{C}(e, \sigma_k, n)\subset \Gamma_{k}(n). $ By definition of hyperbolic polynomial,
$$
\sigma_n(\lambda)=\lambda_1\lambda_2\ldots \lambda_n
$$
is hyperbolic with respect to $e=(1,1,\ldots,1).$ Noticing
$$
\sigma_n(se+\lambda)=\sum_{j=0}^ns^{n-j}\sigma_j(\lambda),
$$
by the convention $\sigma_0(\lambda)=1,$ and
$$
\frac{d^{n-j}}{s^{n-j}}\sigma_n(se+\lambda)\mid_{s=0}=C(n,j)\sigma_j(\lambda).
$$
By using  \eqref{eq:2.1} and Rolle's Theorem again, we see that
$\sigma_k(\lambda), 1\leq k\leq n$ are hyperbolic with respect to
$e=(1,1,\ldots,1)$. Accordingly, by definition \eqref{eq:2.2} of
$\mathscr{C}(e, \sigma_j, n), (1\leq j\leq k)$ and letting $s=0$, we obtain
$$
\sigma_j(\lambda)>0,\ \ \ 1\leq j\leq k.
$$
This completes the proof $\mathscr{C}(e, \sigma_k, n)\subset \Gamma_{k}(n).$ Conversely, if
$\lambda\in \Gamma_{k}(n)$, since
$$
\sigma_k(se+\lambda)=\sum_{j=0}^{k}C(j,k,n)s^{j}\sigma_{k-j}(\lambda),\indent \
C(j,k,n)=\frac{(^{n}_{k})(^{k}_{j})}{(^{n}_{k-j})}
$$
and by definition of $\Gamma_{k}(n)$,
$$
\sigma_{k-j}(\lambda)>0, \ \ \ \ \ 0\leq j\leq k,
$$
 we see that $\sigma_k(se+\lambda)>0$ for all $s\geq0$ and therefore
 $\lambda\in \mathscr{C}(e, \sigma_k, n).$  This completes the proof of $\Gamma_{k}(n)\subset \mathscr{C}(e, \sigma_k, n).$

\noindent
{\bf Step 2.} We will prove that
$\widetilde{\Gamma}_{k}(n)=\Gamma_{k}(n).$ Obviously
$\widetilde{\Gamma}_{k}(n)\subset\Gamma_{k}(n).$ It is left to prove
$\mathscr{C}(e, \sigma_k, n)\subset\widetilde{\Gamma}_{k}(n).$ Let $\lambda\in \mathscr{C}(e, \sigma_k, n),$ we use
\eqref{eq:2.4} with $P(\lambda)=\sigma_j(\lambda) (1\leq j\leq k)$
and $\mu=(1,0,\ldots,0)=e_1$ and then obtain
\begin{equation}\label{eq:2.13}
\frac{\partial \sigma_j}{\partial \lambda_1}(\lambda)=(\nabla
\sigma_j,e_1)\geq 0,\ \ \ 1\leq j\leq k.
 \end{equation}
 We claim that
\begin{equation}\label{eq:2.14}
\frac{\partial \sigma_k}{\partial \lambda_1}(\lambda)>0, \ \ \ \
\forall \lambda \in \mathscr{C}(e, \sigma_k, n).
\end{equation}
Otherwise, we assume that $\frac{\partial \sigma_k}{\partial
\lambda_1}(\lambda^{0})=0$ for some $\lambda^0\in \mathscr{C}(e, \sigma_k, n).$ By
\eqref{eq:2.11},
$$\sigma_{k}(\lambda^0)=\lambda^0_{1}
\sigma_{k-1;1}(\lambda^0)+\sigma_{k;1}(\lambda^0)=\lambda^0_{1}
\frac{\partial \sigma_k}{\partial
\lambda_1}(\lambda^0)+\sigma_{k;1}(\lambda^0)=\sigma_{k;1}(\lambda^0),$$
from which we obtain, by \eqref{eq:2.6},
\begin{equation}\label{eq:2.15}
\sigma_{k-1;1}(\lambda^0)=\sigma_{k-1}(\lambda^0_2,\lambda^0_3,\ldots,\lambda^0_n)
=0,0<\sigma_{k}(\lambda^0)=\sigma_{k;1}(\lambda^0)=
\sigma_{k}(\lambda^0_2,\lambda^0_3,\ldots,\lambda^0_n).
 \end{equation}

Moreover, by virtue of \eqref{eq:2.13}, we obtain
$$
0\leq \frac{\partial \sigma_j}{\partial \lambda_1}(\lambda^0)
=\sigma_{j-1}(\lambda^0_2,\lambda^0_3,\ldots,\lambda^0_n),\ \ 1\leq
j\leq k.
$$
Therefore we have proved that $(\lambda^0_2,\lambda^0_3,\ldots,\lambda^0_n)\in \overline{\Gamma}_k(n-1).$
By \eqref{eq:2.9}, for $\varepsilon>0,$
$$
(\lambda^0_2+\varepsilon,\lambda^0_3+\varepsilon,\ldots,\lambda^0_n+\varepsilon)\in
\Gamma_k(n-1),
$$
to which we apply the Maclaurin's inequalities \eqref{eq:2.7} and obtain
$$
\left[\frac{1}{(^{n-1}_{k})}\sigma_{k}(\lambda^0_2+\varepsilon,\lambda^0_3+\varepsilon,\ldots,\lambda^0_n+\varepsilon)
\right]^{\frac{1}{k}}\leq
\left[\frac{1}{(^{n-1}_{l})}\sigma_{l}(\lambda^0_2+\varepsilon,\lambda^0_3+\varepsilon,\ldots,\lambda^0_n+\varepsilon)
\right]^{\frac{1}{l}}.
$$
Letting $\varepsilon\rightarrow 0^+,$ we have
$$
\left[\frac{1}{(^{n-1}_{k})}\sigma_{k}(\lambda^0_2,\lambda^0_3,\ldots,\lambda^0_n)
\right]^{\frac{1}{k}}\leq
\left[\frac{1}{(^{n-1}_{l})}\sigma_{l}(\lambda^0_2,\lambda^0_3,\ldots,\lambda^0_n)
\right]^{\frac{1}{l}},
$$
which contradicts with \eqref{eq:2.15} in case $l=k-1$, thus the claim \eqref{eq:2.14} is true. Notice that
$$
\frac{\partial \sigma_k}{\partial \lambda_1}(\lambda)=\sigma_{k-1}(\lambda_2,\lambda_3,\ldots,\lambda_n),
$$
then applying the Maclaurin's inequality to \eqref{eq:2.13}-\eqref{eq:2.14} leads to
$$
(\lambda_2,\lambda_3,\ldots,\lambda_n)\in \Gamma_{k-1}(n-1).
$$
Now we can regard that $(\lambda_2,\lambda_3,\ldots,\lambda_n)$  is in the same
position of $(\lambda_1,\lambda_2,\ldots,\lambda_n)$ as above, by an
induction on $k$ we can prove
$$
\sigma_{k-l;i_1,\ldots,i_l}(\lambda)>0,l=0,1,\ldots,k.
$$
That is $\Gamma_k(n)\subset \widetilde{\Gamma}_k(n).$
\end{proof}

For the G{\aa}rding cone in the space of symmetric matrices, similar
results to \eqref{eq:2.14} can be seen Section 3 in \cite{IPY}.

The following is some by-product of \eqref{eq:2.12}. Assume that
$\lambda\in\Gamma_{k}(n)$ is in descending order,
$$
\lambda_{1}\geq \cdots\lambda_{p-1}\geq \lambda_{p}>0\geq
\lambda_{p+1}\geq \cdots\lambda_{n},
$$
then
\begin{equation}\label{eq:2.16}
\left\{
\begin{array}{l}
p\geq k \\
 0<\sigma_{k-1;1}(\lambda)\leq \sigma_{k-1;2}(\lambda)\leq
\cdots\leq \sigma_{k-1,n}(\lambda).
\end{array}\right.
\end{equation}
Otherwise, if $p<k,$ we have
$$
\sigma_{1;\lambda_1,\lambda_2,\ldots,\lambda_{k-1}}=
\sigma_{1}(\lambda_{k},\lambda_{k+1},\ldots,\lambda_n)=\sum_{j=k}^n\lambda_j\leq0,
$$
which contradicts with \eqref{eq:2.12}. Using \eqref{eq:2.12} again,
we have $\sigma_{k-2;12}(\lambda)\geq 0$ and then
\begin{equation*}
\begin{split}
\frac{\partial \sigma_k}{\partial
\lambda_1}(\lambda)&=\sigma_{k-1;1}(\lambda)
=\sigma_{k-1;12}(\lambda)+\lambda_2\sigma_{k-2;12}(\lambda) \\
&\leq\sigma_{k-1;12}(\lambda)+\lambda_1\sigma_{k-2;12}(\lambda)=\sigma_{k-1;2}(\lambda)=\frac{\partial
\sigma_k}{\partial \lambda_2}(\lambda),
\end{split}
\end{equation*}
the remaining part of \eqref{eq:2.16} can be proved similarly. Here
the proof of \eqref{eq:2.16} is adapted  from \cite{LT}.

\smallskip
\noindent {\bf Acknowledgements.} The research of first author is
supported by the National Science Foundation of China No.11171339
and Partially supported by National Center for Mathematics and
Interdisciplinary Sciences. The research of the second author and
the last author is supported partially by ``The Fundamental Research
Funds for Central Universities''  and the National Science Foundation
of China No. 11171261.

\end{document}